\documentclass{amsart}
\usepackage{amssymb,hyperref}
\usepackage{verbatim}
\usepackage{url}

\usepackage{xcolor}

\usepackage[all]{xy}

\numberwithin{equation}{section}

\newtheorem{theorem}{Theorem}[section]
\newtheorem{lemma}[theorem]{Lemma}
\newtheorem{proposition}[theorem]{Proposition}
\newtheorem{corollary}[theorem]{Corollary} 
\theoremstyle{definition}  
\newtheorem{definition}[theorem]{Definition}
\newtheorem{example}[theorem]{Example}
\newtheorem{conjecture}[theorem]{Conjecture}  
\newtheorem{question}[theorem]{Question}
\newtheorem{remark}[theorem]{Remark}

\newcommand{\C}{{\mathcal C}}

\newcommand{\cA}{{\mathcal A}}

\newcommand{\cD}{{\mathcal D}}

\newcommand{\be}{{\bf 1}}
\newcommand{\kk}{{\bf k}}
\newcommand{\ot}{{\otimes}}
\newcommand{\bt}{{\, \boxtimes \,}}

\newcommand{\BZ}{{\mathbb Z}}
\newcommand{\BN}{{\mathbb N}}

\newcommand{\BC}{{\mathbb C}}
\newcommand{\BR}{{\mathbb R}}

\newcommand{\Ve}{\mathrm{Vec}}
\newcommand{\Rep}{\mathrm{Rep}}
\newcommand{\Ver}{\mathrm{Ver}}
\newcommand{\sVe}{\mathrm{sVec}}
\newcommand{\sVec}{\sVe}
\newcommand{\Fr}{{\mathrm{Fr}}}
\newcommand{\Triv}{\mathrm{Triv}}
\newcommand{\KK}{\mathbb K}
\newcommand{\LL}{\mathbb L}
\newcommand{\CC}{{\mathcal C}}

\newcommand{\Dim}{\mathrm{Dim}}

\newcommand{\cO}{\mathcal{O}}

\newcommand{\Frob}{\mathrm{Frob}}
\newcommand{\Sym}{\mathcal{S}ym}

\newcommand{\FPdim}{\mathrm{FPdim}}

\newcommand{\gd}{\mathsf{gd}}
\newcommand{\sd}{\mathsf{sd}}
\newcommand{\ad}{\mathsf{ad}}

\newcommand{\Hom}{\mbox{Hom}}

\newcommand{\Ind}{\mathrm{Ind}}
\newcommand{\id}{\mathrm{id}}
\newcommand{\Id}{\mbox{Id}}
\newcommand{\End}{\mathrm{End}}

\begin{document}

\title{On Frobenius exact symmetric tensor categories}

\author{Kevin Coulembier}
\address{School of Mathematics and Statistics, University of Sydney, Australia}
\email{kevin.coulembier@sydney.edu.au}
\author{Pavel Etingof}
\address{Department of Mathematics, MIT, Cambridge, MA USA 02139
}
\email{etingof@math.mit.edu}
\author{Victor Ostrik}
\address{Department of Mathematics,
University of Oregon, Eugene, OR 97403, USA}
\address{Laboratory of Algebraic Geometry,
National Research University Higher School of Economics, Moscow, Russia}
\email{vostrik@math.uoregon.edu}
\dedicatory{}

\begin{abstract} A fundamental theorem of P. Deligne (2002) states that a pre-Tannakian category over an algebraically closed field of characteristic zero admits a fiber functor to the category of supervector spaces (i.e., is the representation category of an affine proalgebraic supergroup) if and only if it has moderate growth (i.e., the lengths of tensor powers of an object grow at most exponentially). In this paper we prove a characteristic $p$ version of this theorem. 
Namely we show that a pre-Tannakian category over an algebraically closed field of characteristic $p>0$ admits a fiber functor into the Verlinde category $\Ver_p$ (i.e., is the representation category of an affine group scheme in $\Ver_p$) if and only if it has moderate growth and is Frobenius exact. This implies that Frobenius exact pre-Tannakian categories of moderate growth admit a well-behaved notion of Frobenius-Perron dimension.
 
It follows that any semisimple pre-Tannakian category of moderate growth has a fiber functor to ${\rm Ver}_p$ (so in particular Deligne's theorem holds on the nose for semisimple  pre-Tannakian categories in characteristics $2,3$). This settles a conjecture of the third author from 2015. 

In particular, this result applies to semisimplifications of categories of modular representations of finite groups (or, more generally, affine group schemes), which gives new applications to classical modular representation theory. For example, it allows us to characterize, for a modular representation $V$, the possible growth rates of the number of indecomposable summands in $V^{\otimes n}$ of dimension prime to $p$.  

\end{abstract}

\date{\today} 
\maketitle  

\begin{center}
\textbf{To Vera Serganova on her 60th birthday with admiration.}
\end{center}
\tableofcontents

\section{Introduction}
Let $\bold k$ be an algebraically closed field and $\C$ a {\it pre-Tannakian category} 
over $\bold k$, i.e., a $\bold k$-linear abelian symmetric rigid monoidal category with bilinear tensor product and ${\rm End}(\bold 1)={\bold k}$, in which objects have finite length. Assume that $\C$ has {\it moderate growth}, i.e., lengths of $V^{\otimes n}$ for $V\in \C$ 
grow at most exponentially with $n$. If ${\rm char}(\bold k)=0$, a fundamental theorem of P. Deligne (\cite{Del02}, 2002) says that $\C$ admits a {\it fiber functor} $F: \C\to {\rm sVec}$ to the category of supervector spaces, and this functor is unique.\footnote{Conversely, it is obvious that any pre-Tannakian category admitting such a functor has moderate growth.} In other words, $\C$ is the category of representations of the {\it proalgebraic supergroup} $G:=\underline{\rm Aut}_{\otimes}(F)$ on which the parity element $z\in G(\bold k)$ acts by the parity operator (and the pair $(G,z)$ is uniquely determined by $\C$ up to isomorphism). This theorem extended Deligne's earlier theorem (\cite{Del02}, 1990) which characterises Tannakian categories (representation categories of affine group schemes) as those pre-Tannakian categories in which every object has a vanishing exterior power. It is well-known that these theorems do not generalize verbatim to ${\rm char}(\bold k)=p>0$, and their extension to positive characteristic have been an open problem for the last 20 years. The goal of this paper is to propose a (partial) solution to this problem. 

Namely, in our previous works \cite{EOf},\cite{Co} we defined several versions of {\it Frobenius functors} for symmetric tensor categories in characteristic $p$. We say that $\C$ is {\it Frobenius exact} if either of these functors (it does not matter which one) is exact. Any symmetric tensor category which admits a symmetric tensor functor to a semisimple tensor category must be Frobenius exact. 

Now let $\Ver_p$ be the {\it Verlinde category}, i.e., the semisimplification of the category of representations of $\Bbb Z/p$ over $\bold k$ (see \cite{O} and references therein). Following \cite{De}, for $X\in\C$, we denote by $\mathrm{A}^nX$ the image of the skew-symmetrizer of the symmetric group. This is one of the possible notions of an exterior power in a symmetric tensor category. Then our main result is the following theorem, see Corollary~\ref{CorChain} and Theorem~\ref{MainThmText}, which in particular shows that Deligne's criteria for Tannakian and super-Tannakian categories in characteristic zero become equivalent in positive characteristic.   

\begin{theorem}\label{MainThm} For a pre-Tannakian category $\C$, the following conditions are equivalent:
\begin{enumerate}
\item $\C$ admits a symmetric tensor functor $\C \to \Ver_p$;
\item $\C$ is Frobenius exact and of moderate growth;
\item $\C$ is Frobenius exact and for every $X\in\C$, there exists a natural number $n$ for which $\mathrm{A}^nX=0$.
\end{enumerate}
Moreover, a functor as in (1) is unique up to isomorphism when it exists. 
\end{theorem}

Thus any Frobenius exact pre-Tannakian category $\C$ of moderate growth over $\bold k$ 
can be realized as the category of representations of the affine group scheme 
$G:=\underline{\rm Aut}_{\otimes}(F)$ in $\Ver_p$ (compatible with the action of $\pi_1(\Ver_p)$), 
and moreover $G$ is uniquely determined by $\C$.

As a special case of the above main theorem, we obtain a new characterisation of Tannakian categories in positive characteristic. This special case is actually proved in Theorem~\ref{ThmTan2} as a crucial step towards the main theorem:
\begin{theorem}\label{MainThmSpec} For a pre-Tannakian category $\C$, the following conditions are equivalent:
\begin{enumerate}
\item $\C$ is the category of representations of an affine group scheme;
\item $\C$ is Frobenius exact, of moderate growth, and the lax monoidal endofunctor $\Fr_+$ of $\C$ from \cite{Co} is monoidal.
\end{enumerate}
\end{theorem}

\begin{remark}\label{RemMainThm}
\begin{enumerate}
\item Pre-Tannakian categories of moderate growth which are
not Frobenius exact do exist, see \cite{BE, BEO,AbEnv}. Concretely, if $p$ is the characteristic of the base field, then there exists a family $\Ver_{p^n}$ with $n\in\mathbb{Z}_{>0}$ of pre-Tannakian categories, which are only Frobenius exact if $n=1$.

Similarly, Frobenius
exact pre-Tannakian categories which are not of moderate growth can be constructed by using ultrafilter
techniques, see \cite{Harman}. Thus none of two conditions in Theorem~\ref{MainThm}(2)
can be dropped.
\item In the special case of fusion categories (which are automatically Frobenius
exact and of moderate growth), the existence of the functor $\C\to\Ver_p$ was established
in \cite{O}. Also for finite tensor categories (which are automatically of moderate
growth) the equivalence of Theorem~\ref{MainThm} (1) and (2) was established in \cite{EOf}.
The proof in the current paper is independent from \cite{O,EOf}, in particular it does not use lifting theory and the proof in the current paper is in some sense closer in spirit to \cite{Del02}. 
\item Alternative criteria to Theorem~\ref{MainThmSpec} for recognising (super) Tannakian categories are derived in Theorem~\ref{ThmFP}(5) and (6), and in \cite{Co}. \end{enumerate}
\end{remark}

Theorem \ref{MainThm} implies that Frobenius exact pre-Tannakian categories of moderate growth admit a well-behaved notion of {\it Frobenius-Perron dimension} which generalizes 
the usual notion of the Frobenius-Perron dimension in finite tensor categories. 
 
It also follows that any {\it semisimple} pre-Tannakian category over ${\bf k}$ of moderate growth has a fiber functor to ${\rm Ver}_p$ (so in particular Deligne's theorem holds on the nose for semisimple  pre-Tannakian categories in characteristics $2,3$). This settles Conjecture 1.3 in \cite{O}. 

In particular, this result applies to {\it semisimplifications} of categories of modular representations of finite groups (or, more generally, affine group schemes), which gives a new kind of applications to classical modular representation theory. For example, it allows us to characterize, for a modular representation $V$, the possible growth rates of the number of indecomposable summands in $V^{\otimes n}$ of dimension prime to $p$. To the best our knowledge, such results are presently inaccessible via classical methods.

We also give generalizations of our results to not necessarily algebraically closed fields, see Theorem~\ref{ThmAllFields}.  

The organization of the paper is as follows. 

Section 2 contains preliminaries.  

Section 3 reviews and extends the theory of Frobenius functors in symmetric tensor categories developed in our previous works. In particular, we describe a relationship between different kinds of Frobenius functors, which plays an important role in the proof of the main result. 

In Section 4 we introduce three notions of dimensions of objects of symmetric tensor categories (two of them valued in positive integers and one in positive real numbers) and study the relationships between them. This provides technical tools both for the proof of the main theorem and for applications. 

In Section 5 we use the results of Sections 3,4 to show the coincidence of two kinds of Frobenius functors. This enables application of results of \cite{Co} to prove Theorem \ref{MainThm}. 

In Section 6 we describe a stabilization construction which faithfully embeds a symmetric tensor category into one with {\it bijective} Frobenius functor (i.e., an equivalence), called {\it perfect}. If the Frobenius functor is injective to start with, i.e. the category is {\it reduced}, this embedding is fully faithful. This is a categorical analog of the procedure of perfection of a (reduced) commutative $\Bbb F_p$-algebra. 

In Section 7 we combine all the technology of the previous subsections to prove Theorem \ref{MainThm}.  

Finally, Section 8 is concerned with applications. Using Theorem \ref{MainThm}, we extend the theory of Frobenius-Perron dimensions for finite tensor categories to Frobenius exact categories of moderate growth. This leads to new results on the decomposition of tensor powers of a modular representation of a finite group (or, more generally, affine group scheme). 

The paper contains Appendix A by Alexander Kleshchev which proves a result in modular representation theory of symmetric groups that is used in Section 4 to show that condition \ref{CorChain}(3) implies \ref{CorChain}(4). Without Appendix A, this result would follow a posteori from our main result, but only for Frobenius exact categories.

{\bf Acknowledgements.} We are very grateful to Alexander Kleshchev for contributing Appendix A and to Dave Benson for useful discussions. The work of K.C. was partially supported by the ARC grant DP210100251.
The work of P.E. was partially supported by the NSF grant DMS - 1916120. The work of V.O. was partially supported by the NSF grant DMS-1702251 and by the Russian Academic Excellence Project `5-100’.

\section{Preliminaries}

Throughout the paper we let $\kk$ denote an arbitrary field, unless further specified. We denote its characteristic by $\mathrm{char}(\kk)=p\ge 0$, which will be assumed to be strictly positive in most instances.

\subsection{Tensor categories and functors}\label{SecSTC}
A $\kk$-linear monoidal category $(\C,\otimes,\be)$ (with $\kk$-linear tensor product) is a {\em tensor category over $\kk$} if
\begin{enumerate}
\item $\dim_{\kk}\End(\be)=1$;
\item $(\C,\otimes,\be)$ is rigid;
\item $\C$ is abelian.
\end{enumerate}
A {\em symmetric} tensor category (STC) is called {\em pre-Tannakian} if all objects have finite length, which then also implies that morphism spaces are finite dimensional, since the adjunction, by (2), allows to reduce the claim to (1). We refer to \cite{De, EGNO} for the basic theory of tensor categories. However, we stress that our terminology deviates very slightly from \cite{De, Del02}, where tensor categories are always symmetric, or \cite{EGNO}, where tensor categories are always assumed to have objects of finite length. The latter condition is not automatic, see for instance Example~\ref{ExD} below or ultraproducts of pre-Tannakian categories, see \cite{Harman}. 

A $\kk$-linear exact (symmetric) monoidal functor between (symmetric) tensor categories over $\kk$ is called a (symmetric) {\em tensor functor}. We will require a slight generalization. For a field extension $\varphi:\kk\hookrightarrow\KK$, a tensor category $\C$ over $\kk$ and a tensor category $\cD$ over $\KK$, a $\varphi$-tensor functor $\C\to\cD$ is an exact $\kk$-linear monoidal functor. 

The reason for introducing the above terminology is the slightly subtle case where char$(\kk)=p>0$, $\KK=\kk$ and $\varphi$ is the Frobenius homomorphism 
$$\Frob:\;\kk\hookrightarrow\kk, \;\lambda\mapsto \lambda^p.$$ Indeed, we will crucially rely on $\Frob$-tensor functors $\C\to\C$ which are thus not tensor functors in the strict sense.

\subsection{Moderate growth and finite generation}\label{DefModGro}
For an object $V$ in an abelian category, we denote by $\ell(V)\in \mathbb{N}\cup \{\infty\}$, the supremum of the set of all $n\in\mathbb{N}$ for which there exists a strictly ascending chain of subobjects
$$0=V_0\subsetneq V_1\subsetneq V_2\subsetneq\cdots\subsetneq V_n=V.$$

An object $X$ in a tensor category $\C$ has moderate growth if the function which sends $n\in\BN$ to the length of $X^{\otimes n}$ can be bounded by an exponential, {\it i.e.} if $\ell(X^{\otimes n})\le a^n$ for some $a\in\mathbb{R}$. We say that $\C$ {\em is of moderate growth} if every object has moderate growth. For a symmetric tensor category this obviously implies that $\C$ must be pre-Tannakian. For more details on moderate growth we refer to Section~\ref{FinProp}.

A STC $\C$ is called {\em finitely generated} if there exists an object $X\in\C$ such that every object is a subquotient of a direct sum of objects $X^{\otimes i}\otimes (X^{\ast})^{\otimes j}$. In particular the category $\Rep G$ of finite dimensional algebraic representations of an affine group scheme $G$ over $\kk$ is finitely generated if and only if $G$ is of finite type (algebraic).

\subsection{Alternating, divided and symmetric powers}

For an object $X$ in a symmetric tensor category $\C$ over $\kk$ and $n\in\BN$, we define the {\em alternating power} $\mathrm{A}^nX$ as the image of the alternating sum (skew-symmetrizer)
$$a_n:=\sum_{s\in S_n}(-1)^{\mathrm{sgn}(s)}s\;:\;X^{\otimes n}\to X^{\otimes n},$$
with action $\kk S_n\to\End(X^{\otimes n})$ given by the symmetric braiding.
If $p=0$ or $p>n$ then this corresponds to the definition of exterior powers as sign-twisted $S_n$-invariants $\Lambda^nX$ or coinvariants $\wedge^n X$, see \cite[\S 2.1]{EHO}. In general, we have canonical morphisms
$$\wedge^n X\twoheadrightarrow \mathrm{A}^nX\hookrightarrow \Lambda^nX,$$
 which are isomorphisms in the tensor category of vector spaces $\Ve$, but not in the tensor category of supervector spaces $\sVe$ if $p>2$. Note that in \cite{Co, De} $\mathrm{A}^nX$ is denoted by $\Lambda^nX$. We have the traditional formula
\begin{equation}\label{eqAlt}
 \mathrm{A}^n(X\oplus Y)\;\simeq\; \bigoplus_{i=0}^n \mathrm{A}^iX \,\otimes\, \mathrm{A}^{n-i}Y.\end{equation}
 
 We also define $\Gamma^n X$ as the maximal subobject on which the braid action of $S_n$ on $X^{\otimes n}$ becomes trivial. Similarly $S^n X$ can be defined as the coinvariants, or as the degree $n$ summand of the symmetric algebra $SX$, given by the quotient of the full tensor algebra of $X$ by the ideal generated by $\Lambda^2 X$.

 \subsection{Deligne tensor product}
 
 We will use the Deligne tensor product of pre-Tannakian categories from \cite[\S 5]{De} or \cite[\S 4.6]{EGNO}. For non-perfect fields it is an open problem whether the product is always defined, see discussion in \cite[\S 6]{CEOP}, but we will only use a very specific case. Consider pre-Tannakian categories $\C$ and $\cD$, with the latter a semisimple category with 1-dimensional morphism spaces of simple objects ($\cD$ is a direct sum of copies of $\Ve_{\kk}$ as a ${\kk}$-linear category). The pre-Tannakian category $\C\boxtimes \cD=\C\boxtimes_{\kk} \cD$ can be defined as the Karoubi envelope of the ordinary product as $\kk$-linear categories $\C\times_{\kk}\cD$, see \cite[\S 2.2]{O} for details. 
 
 For symmetric tensor functors $F_1:\C\to\C'$ and $F_2:\cD\to\C'$ to a pre-Tannakian category $\C'$ there exists a unique (up to isomorphism) tensor functor 
 $$(F_1,F_2):\C\boxtimes \cD\to\C'$$ which yields $F_1$ and $F_2$ when composed with the canonical tensor functors from $\C$ and $\cD$ to $\C\boxtimes\cD$. If $\C'$ is also of the form $\C_1\boxtimes\cD_1$ we will use the notation $F_1\boxtimes F_2$ for the symmetric tensor functor from $\C\boxtimes\cD$ to $\C_1\boxtimes\cD_1$, canonically obtained from $F_1:\C\to\C_1$ and $F_2:\cD\to\cD_1$ using the above constructions.

 For objects $X\in\C$ and $Y\in\cD$, we will write $X\boxtimes Y$ for the object $(X,Y)\in \C\times_{\kk}\cD$ interpreted as an object in $\C\boxtimes\cD$.

 \subsection{Copseudo-Hopf algebras}
Let $\C$ be a $\kk$-linear abelian category with finite dimensional morphism spaces and objects of finite length (for instance the underlying $\kk$-linear category of a pre-Tannakian category). By Takeuchi's theorem, $\C$ (resp. $\Ind\C$) is equivalent to the category of finite dimensional (resp. all) comodules of a coalgebra $(C,\Delta,\epsilon)$ over $\kk$. Furthermore, for every $n\in\mathbb{Z}_{>0}$, we obtain an equivalence between the category of multifunctors 
$$(\Ind\C)^{\times n}\to\Ind\C$$ which are $\kk$-linear and continuous in each variable and the category of $(C,C^{\otimes n})$-bicomodules. This equivalence sends a bicomodule $B$ to the functor
$$(M_1,M_2,\cdots, M_n)\mapsto B\,\Box^{C^{\otimes n}}\,(M_1\boxtimes M_2\boxtimes \cdots \boxtimes M_n),$$
where $\Box$ denotes the cotensor product, and $\boxtimes$ the external tensor product of two comodules.

The above observation allows one to transfer the categorical structure of a pre-Tannakian category into an algebraic structure on the corresponding coalgebra, which is how we define a {\em cotriangular copseudo-Hopf algebra}. In other words, a cotriangular copseudo-Hopf algebra is the algebraic reformulation of the information contained in the pair of a symmetric tensor category $\C$ and a $k$-linear faithful exact functor $F:\C\to\mathrm{Vec}$, by considering the coend of $F$. In more detail:

\begin{definition}\label{Defpseudo}
A {\em copseudo-bialgebra} consists of a coalgebra $(C,\Delta,\epsilon)$ with
\begin{enumerate}
\item A $(C,C^{\otimes 2})$-bicomodule $T$ which is coflat for both right $C$-coactions and such that $T\Box^{C\otimes C}(M\boxtimes N)$ is finite dimensional for all finite dimensional $C$-comodules $M,N$.
\item A finite dimensional $C$-comodule $E$ with $\dim_{\kk}\End_{C}(E)=1$.
\item An isomorphism of $(C,C^{\otimes 3})$-bicomodules from the bicomodule obtained from $T\otimes T$ by cotensoring over the left coaction of the first copy of $T$ and the right $C\otimes \kk$-coaction of the second copy, to the bicomodule obtained from $T\otimes T$ by cotensoring over the left coaction of the second copy of $T$ and the right $\kk\otimes C$-coaction of the first copy, which satisfies the pentagon axiom.
\item Isomorphisms of $C$-bicomodules $T\Box_l^C E\to C$ and $T\Box_r^C E\to C$, where $T\Box_l^C E$ (resp. $T\Box_r^CE$) takes the cotensor product over $C$ using the $C\otimes\kk$ (resp. $\kk\otimes C$) right coaction on $T$, satisyfing the triangle axiom together with the isomorphism from (3).
\end{enumerate}
This is precisely the structure making the category of finite dimensional $C$-comodules into a monoidal category with exact $\kk$-linear tensor product and $\dim_{\kk}\End(\be)=1$.

A {\em cotriangular copseudo-bialgebra} is a copseudo-bialgebra with additionally
\begin{enumerate}
\item[(5)] An involutive $(C,C^{\otimes 2})$-bicomodule isomorphism $T\xrightarrow{\sim} T'$, where $T'$ is the $(C,C^{\otimes 2})$-bicomodule obtained from $T$ by switching the two right $C$-coactions, satisfying the two hexagon axioms together with the isomorphism from (3).
\end{enumerate}

A copseudo-bialgebra is called a {\em copseudo-Hopf algebra} if the induced monoidal structure on its comodule category is rigid. 
\end{definition}

Explicit algebraic expressions for when a copseudo-bialgebra is a copseudo-Hopf algebra are quite unwieldy. Since they also have limited relevance for the remainder of the paper we stick to the above indirect definition. 

\begin{example}\label{ClassEx}
Let $(H,\Delta,\epsilon)$ be a coalgebra and assume we are given coalgebra morphisms $\nabla: H\otimes H\to H$ and $\eta:\kk\to H$.
\begin{enumerate}
\item Let $T:=H\otimes H$ be the $(H,H^{\otimes 2})$-bicomodule with left $H$-coaction defined via $\nabla$.
\item Let $E=\kk$ be the $H$-comodule defined via $\eta$.
\item Both bicomodules in \ref{Defpseudo}(3) are then isomorphic to $H^{\otimes 3}$. The left $H$-coaction of the first is then defined via $\nabla\circ(\nabla\otimes \Id)$ and that of the second via $\nabla\circ(\Id\otimes\nabla)$. If $\nabla$ is associative, we thus obtain the desired isomorphism.
\item Similarly, if $\eta$ is a unit for the product $\nabla$, we can construct canonical isomorphisms as in \ref{Defpseudo}(4).
\end{enumerate}
This spells out directly how a bialgebra can be interpreted as a copseudo-bialgebra.
\end{example} 

\begin{example}
More generally, a copseudo-bialgebra with $T =H\otimes H$ and $E=\kk$ via morphisms $\nabla$ and $\eta$, as in Example~\ref{ClassEx}, is a {\em coquasi-bialgebra} (or dual quasi-algebra), see \cite[\S 2.4]{Ma}. This is the algebraic structure corresponding to a $k$-linear abelian monoidal category $\C$ with quasi-tensor functor $\C\to\mathrm{Vec}$, see \cite[5.1.1]{EGNO}.
\end{example}

\begin{example}
The dual of a finite dimensional pseudo-bialgebra, see \cite[\S4]{EOf}, is canonically a copseudo-bialgebra.
\end{example}
 
\section{Frobenius functors} In this section we assume that $p>0$.
\subsection{General construction}\label{Frob1} Let $\C$ be a pre-Tannakian category and let $H$ be a cotriangular copseudo-Hopf algebra
with tensor equivalence $H\mbox{-Comod}\simeq\C$. For any finite group $\Gamma$ we have an equivalence of tensor categories:
\begin{equation}
\Rep(\Gamma,\C)\simeq H\mbox{-Comod}(\Rep_{\kk} \Gamma),
\end{equation}
where the left-hand side is the equivariantization with respect to the trivial action of $\Gamma$ on $\C$ and the right-hand side is the category of comodules of $H$ in $\Rep \Gamma=\Rep_{\kk} \Gamma$, where $H$ is a copseudo-Hopf algebra in the ind-completion of $\Rep \Gamma$ via the inclusion of the category of vector spaces.
 
Let $G$ be a subgroup of the symmetric group $S_n$ and let $\overline{\Rep G}$ be the semisimplification of the category $\Rep G$, see e.g. \cite{EOs}. We will make the assumption that simple objects in $\overline{\Rep G}$ have one-dimensional endomorphism algebras, for instance $\kk$ is algebraically closed or $G=C_p$ or $G=S_p$. Following the procedure from \cite{EOf}, we define the
symmetric monoidal functor $\Fr_G: \C \to \C \boxtimes \overline{\Rep G}$ as a composition
$$\C \to \Rep(G,\C)\simeq  H\mbox{-Comod}(\Rep G) \to H\mbox{-Comod}(\overline{\Rep G})=\C \boxtimes \overline{\Rep G}$$
where the first arrow is the functor $X\mapsto X^{\otimes n}$. The second arrow is the semisimplification functor $\Rep G\to\overline{\Rep G}$, which is the identity on objects, sends a morphism to its equivalence class and therefore naturally lifts to a functor between the categories of comodules.

Let $P< G$ be a Sylow $p-$subgroup of $G$ and let $L=N_G(P)$ be its normalizer. Recall from \cite[Proposition~4.1]{EOs} that
the restriction functor $\Rep G\to \Rep L$ induces an equivalence $\overline{\Rep G}\simeq
\overline{\Rep L}$, so the functors $\Fr_G$ and $\Fr_L$ are isomorphic.

\begin{proposition} The functor $\Fr_G$ is additive if and only if the subgroup $P< S_n$
is transitive. In particular, a necessary condition is that $n$ is a power of $p$.
\end{proposition}
\begin{proof}
This is a direct application of Mackey's theorem and \cite[Proposition~4.7]{EOs}.  The requirement that $n$ be a power of $p$ follows from the observation that the Sylow $p$-subgroups of $S_n$ are not transitive when $n$ is not a power of $p$. 
\end{proof}

Now assume that $n=p^j$ is a power of $p$ and $G$ is as in the proposition. 
Then $$\Fr_G: \C\to\C\boxtimes\overline{\Rep G}$$
is easily seen to be $\Frob^j$-linear.

If we assume that $\kk$ is perfect, then we can define the pre-Tannakian category $\C^{(j)}$ over $\kk$ as the $j-$th Frobenius twist of $\C$, which is the same monoidal category, only with new $\kk$-linear structure obtained from the old via restriction along $\Frob^j:\kk\to\kk$.
Then the functor $$\Fr_G: \C \to \C^{(j)}\boxtimes \overline{\Rep G}$$ is $\kk-$linear. Note that the same definition of $\C^{(j)}$ makes sense without the assumption that $\kk$ be perfect, but then $\C^{(j)}$ is not a tensor category over $\kk$, as the endomorphism algebra of $\be\in\C^{(j)}$ does not have dimension $1$.


\subsection{The (enhanced) Frobenius functor}\label{SecEnh} The most important special cases of the construction above are $G=S_p$ and
$G=C_p$ where $C_p< S_p$ is the cyclic subgroup of order $p$.

The semisimplification $\overline{\Rep C_p}$ is called the (universal) Verlinde category $\Ver_p$, see e.g.
\cite[\S 3.2]{O}, and has simple objects $\be=L_1,L_2,\cdots, L_{p-1}$.
We will call $\Fr_{C_p}$ the Frobenius functor (it is called the external Frobenius
twist in \cite{Co}) and denote it by $\Fr$. We will call the functor
$\Fr_{S_p}$ the enhanced Frobenius functor and denote it by $\Fr^{en}$.

Recall that for $p>2$ we have a decomposition of symmetric tensor categories $\Ver_p=\Ver_p^+\bt \sVe$, see \cite[\S 3.2]{O}, where $\Ver_p^+$ is generated by the $L_i$ for $i$ odd and $L_{p-1}$ is the odd line in $\sVe$.
We also have 
$$\Ver_p^{en}:=\overline{\Rep S_p}=\Ver_p^+\bt \Rep(\BZ/2(p-1),z),$$ see \cite[\S 4.4]{EOs}. 
Since $[S_p:C_p]$ is not divisible by $p$, the restriction functor $\Rep S_p\to \Rep C_p$ maps negligible morphisms to negligible ones, see \cite[\S 3.3]{EOs}. Therefore
it induces a well defined functor $R: \overline{\Rep S_p}\to \overline{\Rep C_p}$; it is clear that
\begin{equation}\label{enFr}
\Fr =(\id \bt R)\circ \Fr^{en}.
\end{equation}
It also follows from the analysis in \cite{EOs} that $R$ corresponds to the identity on $\Ver_p^+$ and the forgetful functor $\Rep(\BZ/2(p-1),z)\to\sVe$. In particular, the preimage of $\sVec$ under $R$ is precisely the tensor subcategory of $\Ver_p^{en}$ generated by invertible objects. We denote the latter by $\Ver_p^{en\ast}\simeq \Rep(\BZ/2(p-1),z)$.

For any $\kk$-linear symmetric monoidal abelian category $\cD$ (with $\kk$-linear tensor product) and $j\in\BZ_{>0}$, following \cite[\S 4]{Co}, we can define a $\Frob^j$-linear functor $\Fr_+^{(j)}:\cD\to\cD$, which sends an object $X$ to the image of canonical composition
$$\Gamma^{p^j}X\;\hookrightarrow\; X^{\otimes p^j}\;\twoheadrightarrow\; S^{p^j}X$$
from the $S_{p^j}$-invariants in $X^{\otimes p^j}$ to the coinvariants. If $\cD=\C$ is a pre-Tannakian category, it follows that $\Fr_+=\Fr_+^{(1)}$ is a direct summand of $\Fr^{en}$; namely $$\Fr_+=(\Id \bt \Hom(\be,?))\circ \Fr^{en}.$$
In particular, $\Fr_+$ is then canonically lax (and oplax) monoidal. We will say that `$\Fr_+$ {\em is} monoidal' when this lax monoidal structure is actually monoidal (equivalently, when $\Fr^{en}$ lands in $\C\boxtimes \Ve\subset\C\boxtimes\Ver_p^{en}$).

\begin{definition}\label{DefFT}
Following \cite[\S 3.3]{O}, the {\em Frobenius type} of a Frobenius exact pre-Tannakian category is the minimal tensor subcategory $\cA\subset\Ver_p$ such that $\Fr$ takes values in
$\C\boxtimes\cA\subset \C\boxtimes\Ver_p$. If $p>3$, by \cite[Proposition 3.3]{O}, the tensor subcategories of $\Ver_p$ are $\mathrm{Vec}$, $\Ver_p^+$, $\sVec$ and $\Ver_p$. If $p=2$, we have $\mathrm{Vec}=\Ver_2$ and if $p=3$, the two options are $\mathrm{Vec}=\Ver_3^+$ and $\sVec=\Ver_3$.
\end{definition}

\begin{example}\label{ExFrFun} We spell out the actions of the Frobenius functors on the category $\Ver_p$ for $p>2$:

(i) We have (see \cite[Example~3.8]{O}):
$$\Fr(L_i)=\left\{ \begin{array}{cc} \be \bt L_i& \; \mbox{if} \; i \; \mbox{is odd};\\
L_{p-1}\bt L_{p-i}& \; \mbox{if} \; i \; \mbox{is even}. \end{array}\right.$$
In particular, $\Ver_p$ is of Frobenius type $\Ver_p^+$.

(ii) With $\overline{S}\in\Ver_p^{en*}$ the image of the sign module of $S_p$ in $\Ver_p^{en}$, we have
$$\Fr^{en}(L_i)=\left\{ \begin{array}{cc} \be \bt L_i& \; \mbox{if} \; i \; \mbox{is odd};\\
L_{p-1}\bt (L_{p-i}\boxtimes \overline{S})& \; \mbox{if} \; i \; \mbox{is even}. \end{array}\right.$$
To demonstrate this we focus on the non-trivial case $p>3$.
It then follows from (i) and \eqref{enFr} that, for $i=3$, we have $\Fr^{en}(L_3)=1\bt X$ where
$R(X)=L_3$. Thus $X\in \Ver_p^{en}=\Ver_p^+\bt \Ver_p^{en*}$ should be a product of $L_3$ and some
invertible object. However $L_3$ is a direct summand of $L_3\ot L_3$, so we get that $X=L_3$.
By considering powers of $L_3$, we obtain the result for $i$ odd. For $i$ even it then suffices to observe that $L_i=L_{p-1}\otimes L_{p-i}$.
\end{example}

\subsection{Properties of Frobenius functors}\label{SecFrComm}

We will use the following observation from \cite[Proposition~5.1(ii)]{EOf}. Consider a pre-Tannakian category $\C$, resp. $\cD$, over a field $\KK$, resp. $\LL$, of characteristic $p$.
For a field extension $\varphi:\KK\hookrightarrow {\LL}$ and
a symmetric $\varphi$-tensor functor $H:\C\to\cD$ we have a commutative diagram
$$\xymatrix{
\C\ar[rr]^-{\Fr}\ar[d]^H&&\C\boxtimes_{\KK}(\Ver_p)_{\KK}\ar[d]^{H'}\\
\cD\ar[rr]^-{\Fr}&&\cD\boxtimes_{\LL}(\Ver_p)_{\LL},
}$$
where we use the notation $(\Ver_p)_{\KK}$ to specify that we consider the Verlinde category over ${\KK}$ and the symmetric $\varphi$-tensor functor $H'$ is the unique one which restricts to $H$ on $\C$ and is the identity on objects in $(\Ver_p)_{\KK}$.

\begin{proposition} \label{Frex}
The following conditions are equivalent for a pre-Tannakian category $\C$:

\begin{enumerate}
\item The functor $\Fr$ is exact;
\item The functor $\Fr^{en}$ is exact;
\item The functor $\Fr_+$ is exact;
\item For any non-zero morphism $u: \be \to X$ we have $\Fr_+(u)\ne 0$;
\item There exists a $\kk $-linear faithful exact symmetric monoidal functor $\C\to\cD$ to a $\kk$-linear abelian symmetric monoidal category (with $\kk$-linear tensor product) which sends every short exact sequence in $\C$ to a split one in $\cD$.
\end{enumerate}
\end{proposition}

\begin{proof}  It is clear from \eqref{enFr} that (1) is equivalent to (2). 
A direct summand of an exact functor
is exact, so (2) implies (3). The implication (3) $\Longrightarrow$ (4) is obvious from the
exact sequence $0\to \be \to X$. Property (4) implies that the functor $\Fr^{en}$ is faithful,
so it is exact by \cite[Theorem~2.4.1]{CEOP} and we get (2). Alternatively we can use equivalence of (i) and (iii) in \cite[Theorem~C]{Co}. That (3) and (5) are equivalent is also in \cite[Theorem~C]{Co}.
\end{proof}

\begin{definition}
A pre-Tannakian category $\C$ satisfying the equivalent conditions of Proposition \ref{Frex} is called
{\em Frobenius exact}. 
\end{definition}
\section{Finiteness properties}\label{FinProp}
In this section, we let $\kk$ be a field of arbitrary characteristic. We introduce some generalized notions of `dimensions' of objects in (symmetric) tensor categories over $\kk$. These are $\gd$, $\sd$ and $\ad$, and are based respectively on the growth rate, symmetric group actions and alternating powers (which motivates the notation). Unless specified otherwise, in this section $\C$ is a STC (not necessarily pre-Tannakian).

\begin{definition}\label{Def4}
For $X\in\C$, we define the following elements of $\mathbb{R}_{\ge 0}\cup\{\infty\}$:
\begin{enumerate}
\item $\ell(X)$, the length of $X$, see Section~\ref{DefModGro};
\item $\gd(X)=\lim_{n\to\infty} \ell(X^{\otimes n})^{\frac{1}{n}}$;
\item $\sd(X)$, the supremum of $n\in\mathbb{N}$ for which the natural map $\kk S_n\to\End(X^{\otimes n})$ is injective;
\item $\ad(X)$, the supremum of $n\in\mathbb{N}$ for which $\mathrm{A}^n X\not=0$.
\end{enumerate}
\end{definition}

\begin{remark}
The definitions of $\sd(X)$ and $\ad(X)$ make sense in any $\kk$-linear symmetric monoidal category (for the latter one can reformulate the condition $\mathrm{A}^n X=0$ as the property that the skew-symmetrizer acts as zero on $X^{\otimes n}$). As detailed below, $\gd(X)$ is defined in any tensor category, not necessarily symmetric.
\end{remark}

For a tensor category $\mathcal C$ and $X\in \mathcal C$, set $d_n(X):=\ell(X^{\otimes n})$. Because of rigidity, the tensor product of two (simple) objects cannot be zero. In particular, $\gd(X)=\infty$ as soon as  $d_n(X)=\infty$ for some $n$. Assume now that $d_n(x)<\infty$ for all $n\in\BN$, then
$$
d_{n+m}(X)\ge d_n(X)d_m(X),
$$ 
i.e., the sequence $\lbrace d_n(X)^{-1}\rbrace $ is submultiplicative. Therefore, by Fekete's lemma (\cite{B}, Lemma 1.6.3),
the limit in Definition~\ref{Def4}(2) exists. Moreover, if $X\ne 0$, it follows that 
$$1\le \gd(X)={\rm sup}_{n\ge 1}d_n(X)^{\frac{1}{n}}.
$$
Finally, it follows easily that $\gd(X)<\infty$ if and only if $X$ is of moderate growth.

\begin{example}\label{ExVect}
For a vector space $V$ of dimension $m$, we have
$$\ell(V)=\gd(V)=\sd(V)=\ad(V)=m.$$
For a supervector space $W$ of dimension $m|n$, with $n>0$, we have
$$\ell(W)=\gd(W)=m+n,\; \ad(W)=\begin{cases}\infty&\mbox{ if char$(\kk)=0$}\\
m+(p-1)n&\mbox{if char$(\kk)=p>2$.}
\end{cases}$$
Moreover, by \cite{Berele}, we have $\sd(W)=mn+m+n$ when char$(\kk)=0$. By Lemma~\ref{LemXY2}(2) and Proposition~\ref{Propbcd} below and \cite[\S 5.3]{CEOP}, we also know 
$$m+n\;\le\; \sd(W)\;\le\; m+n +\min(mn, (p-2)m,(p-2)n)$$
when char$(\kk)=p>2$.
\end{example}

\begin{remark}
For given $m,n$, for high enough $p$, we have equality 
$$\sd(W)=m+n+mn= m+n +\min(mn, (p-2)m,(p-2)n).$$
For instance, this can be proved as in \cite{Berele} if $p>(m+1)(n+1)$. 
In general, the inequality in Example~\ref{ExVect} reveals that, contrary to vector spaces, injectivity of ${\kk} S_i\to \End_{{\kk}}(({\kk}^{m|n})^{\otimes i})$ depends on the characteristic of $\kk$.
\end{remark}

\begin{example}\label{ExVerpad}
For an object $X\in \Ver_p$ denote by $X^{\sharp}$ its unique lift in $\Rep C_p$ which does not contain projective direct summands, then $\ad X=\ad X^{\sharp}=\dim_{\kk}X^{\sharp}$. Indeed, $\ad X\le \ad X^{\sharp}$ follows since $\ad$ can only decrease along any symmetric monoidal (not necessarily exact) functor. The other direction follows easily since $\mathrm{A}^i L_i^{\sharp}=\Lambda^i L_i^\sharp\simeq \kk$ for all $i<p$, which is a direct summand of the full $i$-th tensor power. In particular
$$\ad(L_i)\;=\; i\quad\mbox{for $1\le i<p$}.$$
\end{example}


\begin{example}
For the non-trivial simple $X$ and the projective cover $P$ of $\be$ in $\Ver_4$ of \cite{BE, BEO}, see Remark~\ref{RemMainThm}(1), it follows from direct computation that
\begin{eqnarray*}&&\ell(P)=2=\gd(P), \;\sd(P)=3, \;\ad(P)=4,\qquad\mbox{and}\\
&&\ell(X)=1, \;\gd(X)=\sqrt{2},\;\sd(X)=2=\ad(X).
\end{eqnarray*}
\end{example}

\begin{proposition}\label{Propbcd}
For $X\in\C$, we have the following relations between the numbers from Definition~\ref{Def4}:
\begin{enumerate}
\item $\sd(X)\le \ad(X)$;
\item $\ell(X)\le \sd(X\otimes X^\ast)$;
\item $\gd(X\otimes X^\ast)<\infty$ $\Rightarrow$ $\sd(X)<\infty$.
\item $\ell(X)\le\ad(X)$;
\item $\gd(X)\le \ad(X)$.
\end{enumerate}
Moreover, if $\mathrm{char}(\kk)>0$, we have
\begin{enumerate}
\item[(6)] $\sd(X)<\infty$ $\Rightarrow$ $\ad(X)<\infty$
\end{enumerate}
\end{proposition}

We need some preparation for the proof.
\begin{lemma}\label{LemXY}
For $X,Y\in\C$, we have\footnote{Both inequalities $\sd(X\otimes Y)\ge \sd(X)\sd(Y)$ and $\sd(X\otimes Y)\le \sd(X)\sd(Y)$ are false in general. Indeed,  it suffices to consider $\C=\sVec$ in characteristic zero, and $X=Y$ a supervector space of superdimension either $1|1$ or $2|2$.}
$$\gd(X\otimes Y)\ge \gd(X)\gd(Y)\quad\mbox{and}\quad\ad(X\otimes Y)\le \ad(X)\ad(Y).$$
\end{lemma}

\begin{proof}
The property for $\gd$ is straightforward, so we only consider $\ad$. Suppose that $\mathrm{A}^{n+1} X=0=\mathrm{A}^{m+1} Y$, in other words the skew-symmetrizers $a_{n+1}$ and $a_{m+1}$ act trivially on $X^{\otimes n+1}$ and $X^{\otimes m+1}$. We will prove that $a_{nm+1}$ acts trivially on $(X\otimes Y)^{\otimes nm+1}$, from which the claim follows.

 We denote by $C_{s,t}$, for $t\le s$, the quotient of $\kk S_s$ by the two-sided ideal generated by $a_t$, where $S_t$ is interpreted as the subgroup of $S_s$ acting on the first $t$ symbols.

We have a commutative diagram
$$\xymatrix{
\kk S_{nm+1}\ar[rr]\ar[d]^{\Delta}&&\End((X\otimes Y)^{\otimes nm+1})\\
\kk S_{nm+1}\otimes_{\kk} \kk S_{nm+1}\ar[rr]&& \End(X^{\otimes nm+1})\otimes_{\kk}   \End(Y^{\otimes nm+1}),\ar[u]
}$$
where the horizontal arrows correspond to the braiding.

To prove that $a_{nm+1}$ acts trivially on $(X\otimes Y)^{\otimes nm+1}$, it is sufficient to prove that $a_{nm+1}$ is in the kernel of the composition
\begin{equation}\label{EqDelta}\kk S_{mn+1}\xrightarrow{\Delta} \kk S_{mn+1}\otimes_{\kk} \kk S_{mn+1}\twoheadrightarrow C_{mn+1,n+1}\otimes_{\kk} C_{mn+1,m+1}.\end{equation}

Now we apply the commutative diagram to the case where $\C$ is the category of vector spaces and $X,Y$ are vector spaces of dimension $n,m$. Then the right vertical arrow is an isomorphism and the image of the lower horizontal arrow is precisely $C_{mn+1,n+1}\otimes_{\kk} C_{mn+1,m+1}$, see \cite[Theorem~4.2]{DP}. Since $a_{mn+1}$ is in the kernel of the upper horizontal arrow, it must then indeed be in the kernel of \eqref{EqDelta}.
\end{proof}

\begin{lemma}\label{LemXY2}
\begin{enumerate}
\item For an object $W\in\C$ with a finite filtration, and with associated graded object $\mathrm{gr} W$, we have
$$\gd(W)=\gd(\mathrm{gr} W),\;\; \sd(W)\ge \sd(\mathrm{gr} W)\;\mbox{ and }\quad \ad( W)\ge \ad(\mathrm{gr} W).$$
The inequalities are equalities if $p=0$ and in Frobenius exact categories if $p>0$.
\item For objects $X,Y\in\C$, we have
\begin{eqnarray*}
&&\gd(X\oplus Y)\ge \gd(X)+\gd(Y),\;\;  \sd(X\oplus Y)\ge \sd(X)+ \sd(Y)\\
&&\mbox{and}\quad\; \ad(X\oplus Y)=\ad(X)+\ad(Y).
\end{eqnarray*}
\end{enumerate}
\end{lemma}

\begin{proof}
The case $\gd$ in part (1) is immediate. The inequalities for $\sd$ and $\ad$ follow from the following observation.
We have an induced filtration of $W^{\otimes n}$ and, by naturality of the braiding, the action of $\kk S_n$ preserves this filtration. Sending a morphism to its associated graded morphism yields a commutative diagram
$$\xymatrix{
\kk S_n\ar[rr]\ar[rrd]&&\End({\rm gr}W^{\otimes n})\\
&&\End_{\mathrm{filt}}(W^{\otimes n}),\ar[u]
}$$
where $\End_{\mathrm{filt}}$ denotes the set of filtration preserving endomorphisms. That we get an equality for Frobenius exact categories can be proved using Proposition~\ref{Frex}(5). The condition \ref{Frex}(5) always holds in characteristic zero, see \cite{De} or \cite[Remark~3.2.3(i)]{Co}.

Now we prove part (2) for $\gd$. For each $\epsilon>0$ and $Z\in\C$, there is a constant $C_Z(\epsilon)$, such that $d_n(Z)\ge C_Z(\epsilon) (\gd Z-\epsilon)^n$ for all $n$. It then follows that
$$d_n(X\oplus Y)\;\ge \sum_i \binom{n}{i} d_i(X)d_{n-i}(Y)\;\ge\; C_X(\epsilon)C_Y(\epsilon) (\gd X+\gd Y-2\epsilon)^n.$$
Taking $n$-th roots and then the limit yields 
$$\gd(X\oplus Y)\;\ge\; \gd X+\gd Y-2\epsilon.$$
Since this inequality is valid for all $\epsilon>0$, the claim follows. Part (2) for $\sd$ follows from the observation that if for an object $Z$ in $\Rep(H,\C)$ for finite groups $H<G$, we have that $\kk H\to \End(Z)$ is injective, then so is $\kk G\to\End(\mathrm{Ind}^G_HZ)$. The claim for $\ad$ follows from equation~\eqref{eqAlt}.
\end{proof}

Let $\mathcal{S}ym$ be the strict monoidal category, where the objects are labelled by $\BN$, the endomorphism algebra of $i\in\BN$ is $\kk S_i$ (with $\Hom(i,j)=0$ for $i\ne j$), the tensor product on objects is given by addition, and the tensor product on morphisms corresponds to the inclusion $\kk(S_i\times S_j)\hookrightarrow \kk S_{i+j}$.
\begin{lemma}\label{LemSym}
Assume that $\mathrm{char}(\kk)=p>0$. For every non-zero tensor ideal in $\mathcal{S}ym$, there exists $l\in\BN$, such that the symmetrizer $b_l:=\sum_{s\in S_l}s\in \kk S_l$ is in this ideal. The same is true for the skew-symmetrizers (if $p>2$).
\end{lemma}
\begin{proof}
Let $0\not=f\in \kk S_n$ be an element of a non-zero tensor ideal $J$. Let $m\in\BN$ be as in Proposition~\ref{PropSasha} in the appendix. Since $J$ is a tensor ideal, $f$ as interpreted in $\kk S_m$ is still in $J$. Let $P_0$ denote the projective cover of the trivial $\kk S_m$-module, which we interpret as a direct summand of the regular representation. Now there exists $g\in P_0\subset \kk S_m$ such that $fg\not=0$. Since $P_0$ has simple socle spanned by $b_m$, we find that there exists $h\in \kk S_m$ such that $b_m=hfg\in J$. 

For the skew-symmetrizers, we can apply the auto-equivalence of $\Sym$ which is the identity on objects and sends every transposition $\sigma$ in a symmetric group to ${\rm sign}(\sigma)\sigma$.
\end{proof}

\begin{proof}[Proof of Proposition~\ref{Propbcd}]
Part (1) follows by definition. For part (2), we assume that $\ell(X)\ge l$ for some $l\in \BN$. Then $[X\otimes X^\ast:\be]\ge l$. Consider a filtration of $Y:=X\otimes X^\ast$ such that the associated graded ${\rm gr}Y$ contains $\be^l$ as a direct summand. Then we can apply Lemma~\ref{LemXY2}(1).

For part (3) assume that $\sd(X)=\infty$, then 
$$n!=\dim_{\kk} \kk S_n\le\dim_{\kk} \End(X^{\otimes n})=\dim_{\kk}\Hom(\be,Y^{\otimes n}),$$
with $Y=X\otimes X^\ast$. In particular the length of the socle of $Y^{\otimes n}$ is at least $n!$, so
$$\gd(Y) \ge \lim_{n\to \infty}(n!)^{1/n}=\infty,$$
which concludes the proof.

Part (4) follows immediately from both statements regarding $\ad$ in Lemma~\ref{LemXY2}.

 By part (4) and Lemma~\ref{LemXY}, we have
$$\gd(X)\le\lim_{n\to\infty}\ad(X^{\otimes n})^{\frac{1}{n}}\le \ad(X),$$
proving part (5).

Now assume that $\mathrm{char}(\kk)=p>0$. The condition $\sd(X)<\infty$ is equivalent to saying that the symmetric monoidal functor $\Sym\to\C$ which sends the generating object in $\Sym$ to $X$, and realizes the braiding morphisms on the powers of $X$, is not faithful. By Lemma~\ref{LemSym} we then find that indeed $\ad(X)<\infty$.
\end{proof}

\begin{corollary}\label{CorChain}
In a STC $\C$ over a field $\kk$ of positive characteristic, the following conditions are equivalent
\begin{enumerate}
\item $\C$ has moderate growth (so is in particular pre-Tannakian);
\item $\gd(X)<\infty$ for every $X\in\C$;
\item $\sd(X)<\infty$ for every $X\in\C$;
\item $\ad(X)<\infty$ for every $X\in\C$.
\end{enumerate}
\end{corollary}
\begin{proof}
That (1) and (2) are equivalent follows by definition. That (2) implies (3) follows from Proposition~\ref{Propbcd}(3). That (3) implies (4) follows from Proposition~\ref{Propbcd}(6). That (4) implies (2) follows from Proposition~\ref{Propbcd}(5).
\end{proof}	

\begin{remark}\label{Remchar0}
If char$(\kk)=0$, it is still true that (1), (2) and (3) are equivalent in Corollary~\ref{CorChain} and that (4) implies those equivalent conditions. That (2) and (3) are equivalent is precisely \cite[Proposition~0.5]{Del02} and the other claims follow from Proposition~\ref{Propbcd}.
In this case (1), (2), (3) do not imply (4), see Example~\ref{ExVect}.
\end{remark}

\section{Key Lemma}

In this section, we assume that char$(\kk)=p>0$, and we prove that $\Fr_+^{(j)}\simeq \Fr_+^j$ on a Frobenius exact pre-Tannakian category $\C$. Recall that $\Fr^{(j)}_+$ is defined in Section~\ref{SecEnh} on any abelian $\kk$-linear symmetric monoidal category and $\Fr^j_+$ is the composition of $j$ factors $\Fr_+=\Fr_+^{(1)}$.
For every $X\in\C$ and $j>1$, we know that $\Fr_+^{(j)}X$ is a subquotient of $\Fr_+^jX$, see \cite[Lemma~4.1.7]{Co}. To prove the isomorphism between both functors it suffices to show that this subquotient equals the entire object. 



\begin{lemma}\label{ThmFrj}
Let $\C$ be a Frobenius exact pre-Tannakian category, then $\Fr_+^{(j)}\simeq \Fr_+^j$ for all $j>1$.
\end{lemma}
\begin{proof}
By Proposition~\ref{Frex}(5), there exists a $\kk$-linear exact faithful symmetric monoidal functor $F:\C\to\cD$, to an abelian $\kk$-linear symmetric monoidal category $\cD$, so that $F$ splits every short exact sequence. The tensor product on $\cD$ is $\kk$-linear, but not necessarily exact (although it can always be chosen right exact).

 Since $\Fr_+^{(j)}$ and $\Fr_+^j$ intertwine exact symmetric monoidal functors, see \cite[\S 4.1]{Co}, it suffices to prove that  $\Fr_+^{(j)}(F(X_0))$ equals $\Fr_+^j(F(X_0))$ for every $X_0\in\C$. We thus henceforth work in $\cD$ and set $X=F(X_0)$. We use the notation from \cite{Co}, where for a finite group $G$ and an object $M\in\Rep(G,\cD)$, we denote by $\Triv_G(M)$ the image of the $G$-invariants in the $G$-coinvariants. In particular, for $V\in \Rep G$, $\Triv_G(V)$ is the maximal trivial direct summand, and $\Triv_{S_p}(X^{\otimes p})=\Fr_+X$. We will use the following observations, see \cite[2.2.3(i), 2.2.4(i) and 2.2.8]{Co} for a subgroup $G_1<G$:
 \begin{enumerate}
 \item[(i)] If $p$ divides $[G:G_1]$, then $\Triv_G\circ\Ind^G_{G_1}=0$;
 \item[(ii)] If $\Triv_{G_1}M=0$, then $\Triv_{G}M=0$;
\item[(iii)] If $G_1$ contains the normalizer of a Sylow $p$-subgroup of $G$, then $\Triv_G\simeq \Triv_{G_1}$.
 \end{enumerate}

 Denote by $N\subset X^{\otimes p}$ the kernel of $X^{\otimes p}\twoheadrightarrow S^pX$. Then $\Fr_+X$ is isomorphic to the quotient of $\Gamma^pX$ by $\Gamma^pX\cap N$, or equivalently we have a short exact sequence
 $$0\;\to\; N\;\to\; N+\Gamma^pX \;\to\; \Fr_+X\;\to\; 0.$$
This sequence splits, as it is in the image of $F$, and for the same reason $(N+\Gamma^pX )\hookrightarrow X^{\otimes p}$ splits. This allows us to choose a decomposition
$$X^{\otimes p}\;=\; \Fr_+X\oplus Y,\quad\mbox{with} \; Y\simeq  N\oplus X^{\otimes p}/(\Gamma^pX+N), $$
where $Y\subset X^{\otimes p}$ contains $N\subset X^{\otimes p}$, and therefore $Y$ is an $S_p$-subrepresentation. Hence this is a decomposition in $\Rep(S_p,\cD)$, where $\Fr_+X$ is endowed with the trivial $S_p$-action and $\Triv_{S_p}Y$ is zero. Taking the $p^{j-1}$-th tensor power of this equation yields
$$X^{\otimes p^j}\;\simeq\; (\Fr_+X)^{\otimes p^{j-1}}\oplus Z\oplus Y^{\otimes p^{j-1}},$$
where we collect all `mixed terms' in $\Fr_+X$ and $Y$ into the object $Z$. 

We consider the subgroup $H\ltimes L$ of $S_{p^j}$, where $L$ is the Young subgroup $S_p^{\times p^{j-1}}$ of $S_{p^j}$, generated by all permutations $(i,i+1)$ with $i$ not divisible by $p$, and $H$ is the copy of $S_{p^{j-1}}$ inside $S_{p^j}$ of $\sigma\in S_{p^j}$ for which $\sigma(ap+i)=\sigma(ap+1)+i-1$ for all $0< i\le p$ and $0\le a< p^{j-1}$.
The above decomposition is then a decomposition in $\Rep(H\ltimes L,\cD)$. In particular, with $d:=p^{j-1}$, as an $H$-representation $Z$ is
$$Z\;=\;\bigoplus_{i=1}^{d-1}\Ind^{S_{d}}_{ S_i\times S_{d-i}}\left((\Fr_+X)^{\otimes i}\otimes Y^{\otimes d-i}\right).$$
It then follows from the above property (i) that $\Triv_{H}Z=0$, so in particular $\Triv_{H\ltimes L}Z=0$ by property (ii).
Next, we claim that
$$\Triv_L(Y^{\otimes p^{j-1}})\;\simeq\; ( \Triv_{S_p}Y)^{\otimes p^{j-1}}=0.$$
Indeed, the right equality is clear. The left isomorphism follows since we can compute it in $\C$ (as $Y$ is in the essential image of $F$) where the tensor product is exact. Subsequently, (ii) also implies that $\Triv_{H\ltimes L}(Y^{\otimes p^{j-1}})=0$.


Since $L$ acts trivially on $(\Fr_+X)^{\otimes p^{j-1}}$, we can thus conclude that
$$\Triv_{H\ltimes L}(X^{\otimes p^j})\;\simeq\;\Triv_H((\Fr_+X)^{\otimes p^{j-1}})=\Fr_+^{(j-1)}(\Fr_+ X).$$

Now $H\ltimes L<S_{p^j}$ contains $Q_j$ from \cite[\S 1.1]{Co} as a subgroup. As it is observed {\it loc. cit.} that the latter contains the normalizer of a Sylow subgroup of $S_{p^j}$, it follows from (iii) that 
$$\Triv_{H\ltimes L}(X^{\otimes p^j})=\Fr_+^{(j)}(X).$$

We conclude by iteration that $\Fr_+^{(j)}(X)\simeq \Fr_+^jX$.
 \end{proof}


\begin{remark}
The property $\Fr_+^{(j)}\simeq \Fr_+^j$ fails without the assumption that $\C$ be Frobenius exact. Indeed, for $p=2$, one can verify this in the category $\Ver_8^+=\C_3\subset\Ver_8$ from \cite{BE, BEO}, see Remark~\ref{RemMainThm}(1). Concretely, for $X$ the length 2 object with socle $\be$ and top the other simple object, we have $\Fr_+^2X=\be$ and $\Fr_+^{(2)}X=0$.
\end{remark}


\begin{theorem} \label{ThmTan2}
A pre-Tannakian category $\C$ is Tannakian if and only if the following conditions hold:

1) $\C$ is Frobenius exact;

2) The functor $\Fr_+$ is monoidal; 

3) $\C$ is of moderate growth.
\end{theorem}
\begin{proof}
It is well-known that Tannakian categories satisfy these conditions. By \cite[Theorems~5.1.2 and~6.1.1]{Co}, a tensor category $\C$ is Tannakian if for every $X\in\C$, we have (a) $\ad(X)<\infty$ and (b) if $\mathrm{A}^n\Fr^{(j)}_+(X)=0$ then $\Fr^{(j)}_+(\mathrm{A}^nX)=0$, for all $j,n\in\BZ_{>0}$. Under assumptions 1) and 2), condition (b) is trivially satisfied since, by Lemma~\ref{ThmFrj}, $\Fr^{(j)}_+$ is a symmetric tensor functor. The conclusion thus follows from the equivalence between (1) and (4) in Corollary~\ref{CorChain}. 
\end{proof}

If $p=2$, $\Fr_+=\Fr$ is always monoidal and Theorem~\ref{ThmTan2} becomes:
\begin{corollary}\label{peq2}
Assume $p=2$. A pre-Tannakian category $\C$ is Tannakian if and only if the following conditions hold:

1) $\C$ is Frobenius exact;

2) $\C$ is of moderate growth.
\end{corollary}

\section{Stabilization}
Assume $p>0$.
Consider a Frobenius exact pre-Tannakian category $\C$ equipped with a symmetric tensor functor $I:\Ver_p\to \C$ (which is automatically fully faithful). Then we can construct a $\Frob$-tensor functor, see Section~\ref{SecSTC}, as the composition
\begin{equation}\label{ItFr}
\Fr[I]:\;\,\C \xrightarrow{\Fr} \C\boxtimes\Ver_p \xrightarrow{(\id, I)}  \C.
\end{equation}

\begin{definition}\label{DefPerfect}
\begin{enumerate}
\item We say that $(\C,I)$ is {\bf reduced} if $\Fr[I]$
is fully faithful.
\item We say that $(\C,I)$ is {\bf perfect} if $\Fr[I]$
is an equivalence (of additive monoidal categories)
\end{enumerate}
\end{definition}
The above terminology is motivated in Remark~\ref{RemTerm} below.
Clearly, a necessary condition for $(\C,I)$ to be reduced (or perfect) is that the base field $\kk$ be perfect.

\begin{example}\label{VerPerf}
$(\Ver_p,\id)$ is perfect if $\kk$ is perfect. Indeed, by Example~\ref{ExFrFun}(i) we have $\Fr[\id]\simeq\id$ (as additive functors).
\end{example}

For the next theorem we need a construction of 2-colimits of categories (in the special case when the indexing
category is a filtrant category associated with the partially ordered set $\BZ_{\ge 0}$), see e.g. \cite{DS}. We recall
this construction now. Let us consider a chain of categories and functors
\begin{equation}\label{chainC}\C_0\xrightarrow{F_0}\C_1\xrightarrow{F_1}\C_2\xrightarrow{F_2}\cdots.\end{equation}
By abuse of notation we will denote by $F^r$ the composition $F_{a+r-1}\circ \cdots \circ F_{a+1}\circ F_{a}:
\C_a\to \C_{a+r}$ for any $a,r\in \BZ_{\ge 0}$; in particular $F^0=\Id$. For any objects $X,Y \in \C_a$ we have an inductive system
of sets
$$\Hom_{\C_a}(X,Y)\xrightarrow{F} \Hom_{\C_{a+1}}(FX,FY)\xrightarrow{F} \Hom_{\C_{a+2}}(F^2X,F^2Y)\xrightarrow{F} \cdots$$
and let $\varinjlim \Hom(X,Y)$ denote its limit. Thus any element $f\in \varinjlim \Hom(X,Y)$ can be represented
by some $\widetilde f\in \Hom_{\C_{a+r}}(F^rX,F^rY)$ and $F^k(\widetilde f)$ represents the same element of $\varinjlim \Hom(X,Y)$
for any $k\in \BZ_{\ge 0}$. In particular, for any $f\in \varinjlim \Hom(X,Y)$ we can always find a representative $\widetilde f\in \Hom_{\C_{a+r}}(F^rX,F^rY)$ with $r$ greater than any given bound.

We now define a new category $\varinjlim \C_i$ as follows:

1) The class of objects of  $\varinjlim \C_i$ is the disjoint union $\sqcup_i\mathrm{Ob}\C_i$;

2) For $X\in \C_a$ and $Y\in \C_b$ we define 
$$\Hom(X,Y)=\left\{ \begin{array}{cc}\varinjlim \Hom(F^{b-a}X,Y)&\; \mbox{if}\; b\ge a,\\
\varinjlim \Hom(X,F^{a-b}Y)&\; \mbox{if}\; a\ge b.\end{array}\right.
$$
The composition of morphisms $f\in \Hom(X,Y)$ and $g\in \Hom(Y,Z)$ is defined as follows:
for sufficiently large $s\in \BZ_{\ge 0}$
find representatives $\widetilde f\in \Hom_{\C_{b+s}}(F^rX,F^sY)$ and $\widetilde g\in \Hom_{\C_{b+s}}(F^sY,F^tZ)$ of $f$ and $g$ 
and define $g\circ f$ to be represented by $\widetilde g\circ \widetilde f$.

3) We have obvious functors $\C_i \to \varinjlim \C_i$ sending any object to itself and any morphism 
$\widetilde f \in \Hom_{\C_i}(X,Y)$ to the morphism $f\in \Hom(X,Y)$ represented by $\widetilde f$.   

Here are some general remarks:

4) Assume that the categories $\C_i$ and the functors $F_i$ are additive. Then it is clear that $\varinjlim \C_i$
is also additive; also the functors $\C_i \to \varinjlim \C_i$ are additive.
 If the categories $\C_i$ are abelian and the functors $F_i$ are exact then the category
$\varinjlim \C_i$ is abelian: for example the kernel of a morphism $f\in \Hom(X,Y)$ represented by $\widetilde f\in \Hom_{\C_{b+s}}(F^rX,F^sY)$ is the kernel of $\widetilde f$ considered as an object of $\varinjlim \C_i$;
the functors $\C_i \to \varinjlim \C_i$ are exact. 
If the categories $\C_i$ are linear over fields $\kk_i$ and the functors $F_i$ are $\varphi_i-$linear
with respect to field homomorphisms $\varphi_i: \kk_i\to \kk_{i+1}$ then the category $\varinjlim \C_i$ is
linear over $\varinjlim \kk_i$. Indeed, for any $\lambda \in \varinjlim \kk_i$ and $f\in \Hom(X,Y)$ we can find some 
$s\in \BZ_{\ge 0}$ such that $\lambda$ is represented by some $\widetilde \lambda \in \kk_s$ and $f$ is represented by 
$\widetilde f \in \Hom_{\C_s}(F^tX,F^rY)$, and we define $\lambda f$ to be represented by $\widetilde \lambda \widetilde f$. 
The functors $\C_i \to \varinjlim \C_i$ are linear with respect to the natural homomorphisms
$\kk_i \hookrightarrow \varinjlim \kk_i$.

5) Assume that the categories $\C_i$ and the functors $F_i$ are symmetric monoidal. Then the category
$\varinjlim \C_i$ is also symmetric monoidal. Namely for $X\in \C_a$ and $Y\in \C_b$ we define
$$X\otimes Y=\left\{ \begin{array}{cc}X\otimes F^{a-b}Y&\; \mbox{if}\; a\ge b,\\
F^{b-a}X\otimes Y&\; \mbox{if}\; b\ge a\end{array}\right.
$$
with obvious associativity, commutativity and unit constraints. It is clear that the functors $\C_i \to \varinjlim \C_i$
have natural structures of symmetric monoidal functors.

6) Now assume that \eqref{chainC} is a sequence of (symmetric) $\varphi_i$-tensor functors $\C_i\to \C_{i+1}$ between (symmetric) tensor categories over fields $\kk_i$.
By 4) and 5), it then follows $\varinjlim\C_i$ is a (symmetric) tensor category over $\varinjlim \kk_i$.

\begin{lemma}\label{ModLim}
If for a system $\{\C_i\,|\,i\in\mathbb{N}\}$ of pre-tannakian categories $\C_i$ over fields $\kk_i$, as above, each $\C_i$ is of moderate growth, then the symmetric tensor category $\varinjlim\C_i$ over $\varinjlim \kk_i$ is a pretannakian category of moderate growth.
\end{lemma}
\begin{proof}
It is clear that if each $\C_i$ satisfies property \ref{CorChain}(3) then the same is true for $\varinjlim\C_i$. Then we can use equivalence between (1) and (3) in Corollary~\ref{CorChain} (or Remark~\ref{Remchar0} if the fields have characteristic $0$). 
\end{proof} 

It is known, see the following example, that contrary to Lemma~\ref{ModLim} the direct limit  $\varinjlim\C_i$ of pre-Tannakian categories $\C_i$ (without moderate growth) need not be pre-Tannakian.

\begin{example}\label{ExD}
We refer the reader to \cite[2.19]{De} for the details for the above construction applied to the chain
$$\Rep(GL_t)\xrightarrow{F}\Rep(GL_{t-1})\xrightarrow{F}\Rep(GL_{t-2})\xrightarrow{F}\cdots$$
of semisimple pre-Tannakian categories over $\BC$ (with $t\in\BC\backslash \BZ$), where the functors $F: \Rep(GL_{t-i})\to \Rep(GL_{t-i-1})$ are restriction functors sending the generating object to a direct sum of the generating object and $\be$. The resulting tensor category $\varinjlim \Rep(GL_{t-i})$ over $\BC$ has
infinite dimensional $\Hom$ spaces, so is {\em not} pre-Tannakian.
\end{example}

\begin{theorem}\label{ThmPerf}
\begin{enumerate}
\item If $(\C,I)$ is reduced, then 
$$\Fr:\C\to\C\boxtimes\Ver_p\quad\mbox{ and }\quad\Fr^{en}:\C\to\C\boxtimes\Ver_p^{en}$$ are fully faithful. Moreover, if $(\C,I)$ is perfect, then $\Fr$ and $\Fr^{en}$ send simple objects to simple objects.
\item For a Frobenius exact pre-Tannakian category $\C$ over $\kk$, denote by $\KK$ the perfection $\varinjlim \kk$ of $\kk$ with canonical inclusion $\iota:\kk\hookrightarrow \KK$. If $\C$ is of moderate growth, there exists a Frobenius exact pre-Tannakian category $\C'$ of moderate growth over $\KK$ with a $\iota$-tensor functor $\C\to\C'$ such that $(\C',I)$ is perfect, for some tensor functor $I:(\Ver_p)_\KK\to\C'$.
\end{enumerate}
\end{theorem}
\begin{proof}
For part (1), we prove the properties for $\Fr^{en}$, the case of $\Fr$ is easier. We can expand \eqref{ItFr} further to get
$$\C \xrightarrow{\Fr^{en}} \C \bt \Ver_p^{en} \xrightarrow{\id\boxtimes R} \C\boxtimes\Ver_p \xrightarrow{(\id, I)} \C. $$
Every functor in this sequence is a tensor functor, and hence faithful and exact. If the overall composite is fully faithful, the first functor must be fully faithful. If the composite is an equivalence, the first functor must send simple objects to simple objects.

For (2), consider first the tensor functor $\C\to\C\boxtimes\Ver_p$ from the definition of the Deligne tensor product. This shows that it suffices to prove the claim for a category $\C$ which is already equipped with a tensor functor $J:\Ver_p\to\C$.

Then we define $\C'$ as the direct limit of the system
$$\C\xrightarrow{\Fr[J]}\C\xrightarrow{\Fr[J]}\C\xrightarrow{\Fr[J]}\cdots.$$

As summarized in 6) above, $\C'$ is a pre-Tannakian of moderate growth. The limit of $J:\Ver_p\to \C$ yields a tensor functor $I$ from $(\Ver_p)_{\KK}\simeq \varinjlim\Ver_p$ to $\C'$. Clearly, the limit of $\Fr[J]$ along itself becomes an equivalence, so it suffices to demonstrate that this limit is isomorphic to the endofunctor $\Fr[I]$ of $\C'$. The latter follows from commutativity of the exterior rectangle in the diagram
$$
\xymatrix{
\C\ar[r]^-{\Fr}\ar[d]&\C\boxtimes\Ver_p\ar[r]^-{(\id, J)}\ar[d]&\C\ar[d]\\
\C'\ar[r]^-{\Fr}&\C'\boxtimes_{\KK}(\Ver_p)_{\KK}\ar[r]^-{(\id, I)}&\C'.
}
$$
In the above diagram, commutativity of the right square is immediate. Commutativity of the left square follows from commutativity of 
$$
\xymatrix{
\C\ar[r]^-{\Fr}\ar[d]^-{\Fr[J]}&\C\boxtimes\Ver_p\ar[d]^{\Fr[J]\boxtimes \Fr[\id]}\\
\C\ar[r]^-{\Fr}&\C\boxtimes\Ver_p,
}
$$
see Section~\ref{SecFrComm}.
\end{proof}

\begin{remark}\label{RemTerm}
For $\cA$ the Frobenius type of $\C$ (see Definition~\ref{DefFT}) one can define $\Fr[I]$ via some $I:\cA\to\C$ and extend Definition~\ref{DefPerfect}. Let $\kk$ be a perfect field and $G$ an affine group scheme over $\kk$ and consider $I:\mathrm{Vec}\to\Rep G$.  Then $(\Rep G,I)$ is reduced if and only if the scheme $G$ is reduced, and $(\Rep G,I)$ is perfect if and only if $G$ is perfect (as an $\mathbb{F}_p$-scheme).
\end{remark}

Based on the above observation, we have the following examples of the (analog of the) limit procedure in the proof of Theorem~\ref{ThmPerf}.
\begin{example} Let $\kk$ be a perfect field and consider the multiplicative group $\mathbb{G}_m$ over $\kk$ and the group $\mu_p$ of $p$-th roots of unity.
\begin{enumerate}
\item If $\C=\Rep\mathbb{G}_m\simeq\Ve_{\BZ}$ is the category of $\BZ$-graded vector spaces, which is reduced but not perfect, then the limit along $\Fr[I]\simeq\Fr_+$ is $\Ve_{\BZ[1/p]}$.
\item If $\C=\Rep\mu_p\simeq\Ve_{\BZ/p}$ is the category of $\BZ/p$-graded vector spaces, which is not reduced, then the limit along $\Fr[I]\simeq\Fr_+$ is $\Ve$.
\end{enumerate}
\end{example}
\section{Proof of the main theorem}
In this section we assume $p>0$.
\subsection{The subcategory $\C^\ast$.} 

\begin{definition} Let $\C$ be a Frobenius exact pre-Tannakian category. If $p\le 3$ set $\C^\ast=\C$ and if $p\ge 3$ let $\C^*\subset \C$ be the full subcategory consisting of objects $X\in \C$ such
that $\Fr(X)$ lies in subcategory $\C \bt \sVec\subset \C \bt \Ver_p$.
\end{definition}

 Note that the two definitions agree for $p=3$ since $\Ver_3=\sVec$.
The following result is immediate.

\begin{proposition} The subcategory $\C^*\subset \C$ is a Serre subcategory closed under tensor product
and duality.
\end{proposition}

In particular, $\C^*$ is a tensor subcategory of $\C$ in the sense of \cite[4.11.1]{EGNO}. In fact, $\C^\ast$ is the maximal tensor subcategory of $\C$ which is of Frobenius type $\Ve$ or $\sVec$.

\begin{question}
As will follow from Theorem~\ref{ThmAllFields}, Frobenius exact pre-Tannakian categories of moderate growth are either of Frobenius type $\Ve$ or $\Ver_p^+$. For symmetric fusion categories and more generally finite tensor categories, this was already observed in \cite{O, EOf}.  It is an open question whether the same is still true without the assumption on moderate growth.
\end{question}

\begin{example} We have $\Ver_p^*=\sVe \subset \Ver_p$, if $p>2$.
The category $\Ver_p^{en*}$ is the tensor subcategory of $\Ver_p^{en}$ generated by invertible objects, as is consistent with the definition of $\Ver_p^{en*}$ in Section~\ref{SecEnh}.
\end{example}

\subsection{The enhanced Frobenius functor on $\C^\ast$} Assume that $\C=\C^*$ (for a Frobenius exact pre-Tannakian category $\C$). It follows immediately from equation~\eqref{enFr} and the definition of $\Ver_p^{en\ast}$ in Section \ref{SecEnh} that the enhanced Frobenius functor $\Fr^{en}$ then takes values in $\C\boxtimes\Ver_p^{en*}$.

Let $T$ be an isomorphism class of invertible ({\em i.e.} simple) objects in $\Ver_p^{en*}$.
We say that an object $X\in \C$ is homogeneous of degree $T$ if $\Fr^{en}(X)\simeq Y\bt T$ for
some $Y\in \C$. It is clear that the homogeneous objects of degree $T$ form a Serre subcategory
$\C_T\subset \C$  and that $\C_T\otimes \C_S\subset \C_{T\ot S}$.

\begin{proposition} \label{graded prop}
Assume that $\C=\C^\ast$ and that $\Fr^{en}$ is fully faithful. Then 
$$\C =\bigoplus_T\C_T.$$
Thus the category $\C$ is graded (possibly non-faithfully) in the sense of \cite[Section~4.14]{EGNO} by the group of isomorphism
classes  of invertible objects of $\Ver_p^{en*}$ (isomorphic to $\BZ/2(p-1)$ if $p>2$).
\end{proposition}

\begin{proof} Since $\Fr^{en}$ is fully faithful, it maps indecomposable objects to indecomposable ones.
Thus any indecomposable object is homogeneous of some degree and the result follows.
\end{proof} 

Let $\C_\be$ be the neutral component of the grading given by Proposition \ref{graded prop}.
Thus $\C_\be$ consists of objects $X$ such that $\Fr^{en}(X)=Y\bt \be$. In other words, we have
$\Fr^{en}(X)=\Fr_+(X)\bt \be$ for all $X\in \C_\be$ and the functor $\Fr_+$ restricted to $\C_\be$
is a symmetric tensor functor. By Theorem~\ref{ThmTan2} we obtain the following corollary.

\begin{corollary} \label{Tann cor}
Assume furthermore that $\C$ is of moderate growth. Then the subcategory $\C_\be$ is Tannakian.
\end{corollary}

\begin{remark}
It will follow a posteori from Theorem~\ref{ThmAllFields} that the grading in Proposition~\ref{graded prop} corresponds to a faithful $\BZ/2$-grading, if $\C$ is of moderate growth. Without moderate growth it is an open question whether this is true.
\end{remark}

\subsection{Graded categories}
\begin{proposition}\label{PropGradTan} Let $\C$ be pre-Tannakian and graded by an abelian group $H$:
$$\C =\bigoplus_{h\in H}\C_h.$$
Assume that the neutral component $\C_1$ is Tannakian. Then there exists a field extension
$\varphi:\kk \hookrightarrow \KK$ and a symmetric $\varphi$-tensor functor $\C \to \cD$ where $\cD$ is a semisimple pointed category
over $\KK$. \end{proposition}

\begin{proof} By the definition of Tannakian categories, $\C_1$ admits a fiber functor $S$ to $\Ve_{\KK}$ for some
field extension $\kk \subset {\KK}$.  We can and will assume that ${\KK}$ is algebraically closed (if ${\KK}$ is not
algebraically closed, we replace the fiber functor by its composition with $?\otimes_{\KK} \overline {\KK}: \Ve_{\KK} \to \Ve_{\overline {\KK}}$). 
Let $A$ be an (injective) ind-object of $\C_1$ representing the functor $X\mapsto S(X)^*$, that is
$$\Hom_{\C_1}(X,A)=S(X)^*.$$
Then tracing the product of identity morphisms on $A$ through
$$\End(A)\otimes_{\KK}\End(A)= S(A)^\ast\otimes_{\KK} S(A)^\ast\to (S(A)\otimes_{\KK}S(A))^\ast\xleftarrow{\sim}S(A\otimes A)^\ast=\Hom(A\otimes A,A)$$
equips $A$ with the structure of a commutative algebra in the category $\Ind(\C_1)$ with the following properties:


1) for any $X\in \C_1$ the $A-$module $X\ot A$ is isomorphic to a finite direct sum of copies of $A$;

2) $\Hom_{\Ind(\C_1)}(\be,A)={\KK}$ is an algebraically closed field.

Let $\cD_1$ be the full subcategory of the category of $A-$modules in $\Ind(\C_1)$ consisting of direct summands
of objects of the form $X\ot A$ where $X\in \C_1$. It is clear that $\cD_1$ is a rigid symmetric monoidal category;
property 2) above says that $\End_{\cD_1}(\be)={\KK}$ and property 1) implies that the functor $M\mapsto
\Hom_{\cD_1}(\be,M)$ gives a symmetric monoidal equivalence $\cD_1\simeq \Ve_{\KK}$. 

Now let us consider the category $\cD$ which is the full subcategory of the category of rigid $A-$modules in $\Ind(\C)$
consisting of direct summands of objects of the form $X\ot A$ with $X\in \C$. It is clear that $\cD$ is a
${\KK}-$linear rigid symmetric monoidal category (of course a priori the category $\cD$ might fail to be abelian). It is clear that the category $\cD$ inherits an $H-$grading
from the category $\C$:
$$\cD=\bigoplus_{h\in H}\cD_h,$$
and the neutral component of this grading is precisely the category $\cD_1$ above.
Furthermore, the $\Hom-$spaces in the category $\cD$ are finite dimensional vector spaces over $\KK$, since for homogeneous
$X,Y \in \cD$ we have
$$\Hom_\cD(X,Y)=\left\{ \begin{array}{cc}0&\mbox{if the degrees of}\; X\; \mbox{and}\; Y\; \mbox{are distinct};\\
\Hom_{\cD_1}(\be, X^*\ot Y)&\mbox{otherwise}\end{array}\right.$$
Let $X$ be an object of $\cD_h$; we claim that its endomorphism algebra is isomorphic to some matrix
algebra over $\KK$. Pick a non-zero object $Y\in \cD_{h^{-1}}$. By the previous displayed equation, the canonical algebra morphism
$$\End_{\cD}(X)\ot_{\KK}\End_{\cD}(Y)\to\End_{\cD}(X\ot Y)$$
is an isomorphism.
Since $X\ot Y\in \cD_1\simeq \Ve_{\KK}$ we see that the algebra $\End_{\cD}(X\ot Y)$ is a matrix algebra 
over $\KK$; this implies that $\End_{\cD}(X)$ is simple and hence also a matrix algebra (as $\overline{\KK}=\KK$). It follows that all indecomposable objects
of $\cD_h$ are isomorphic to each other and invertible. Thus the category $\cD$ is semisimple (and hence
abelian) and pointed.

Finally $?\ot A: \C \to \cD$ is the required symmetric tensor functor, concluding the proof.\end{proof}

\begin{corollary} \label{graded str cor} 
 Let $\C$ be a graded pre-Tannakian category such that its neutral component is Tannakian.
Then there exists a field extension $\kk \subset \KK$ and a super fiber functor $\C \to \sVe_{\KK}$.
\end{corollary}
\begin{proof}
By \cite[Example~6.2.3]{Co}, every semisimple pointed category admits a super fiber functor. The conclusion thus follows from Proposition~\ref{PropGradTan}.
\end{proof}

\subsection{Concluding the proof} 
In this section we consider a Frobenius exact pre-Tannakian category $\C$ equipped with a tensor functor $I:\Ver_p\to\C$.

\begin{proposition} \label{star prop}
Assume that $(\C,I)$ is reduced. Then the Frobenius functor
$\Fr$ lands in $\C^*\bt \Ver_p$.
\end{proposition}

\begin{proof} Since $\C^*$ is a Serre subcategory, it is sufficient to prove that for a simple object
$X\in \C$, we have $\Fr(X)\in \C^*\bt \Ver_p$. Since $\Fr(X)$ is indecomposable, see Theorem~\ref{ThmPerf}(1),
we have $\Fr(X)=Y\bt L_i$ for some $1\le i< p$ and indecomposable $Y\in \C$.
Thus $\Fr[I](X)=Y\ot I(L_i)$ and $\Fr(Y)=Z\bt L_j$ for some $1\le j< p$.
Thus $(\Fr[I])^2(X)=Z\ot I(L_j\ot L_i)$, see Example~\ref{VerPerf}. Since $\Fr[I]$ is fully faithful, the latter object must be indecomposable, so in particular $L_j\ot L_i$ must be simple. This forces either $i$ or $j$ to be in $\{1,p-1\}$. If $j\in\{1,p-1\}$, we are
done by the definition of $\C^*$. If $i\in\{1,p-1\}$ then $X\in \C^*$, so $Y\in \C^*$ as a subquotient of 
$X^{\ot p}$.
\end{proof}

\begin{proposition}\label{PropRed}
Assume that $(\C,I)$ is reduced and $\C$ is of moderate growth. 
Then there is a field extension
$\kk \subset \KK$ and a fiber functor $\C \to (\Ver_p)_{\KK}$. 
\end{proposition}

\begin{proof} By Proposition \ref{star prop} we have a $\Frob$-tensor functor $\C\to\C^*\bt \Ver_p^{en}$. The category $\C^*$ clearly satisfies $(\C^*)^*=\C^*$, has the property that $\Fr^{en}$ is fully faithful by Theorem~\ref{ThmPerf}(1) and is of moderate growth, so by Proposition \ref{graded prop} it is graded
with Tannakian neutral component, see Corollary \ref{Tann cor}. Thus by Corollary \ref{graded str cor} 
we have a fiber functor $\C^* \to \sVe_{\KK}$ for some field extension $\kk \subset \KK$. Summarizing,
we have the following composition of symmetric monoidal functors:
$$\C \to \C^*\bt \Ver_p\to (\sVe \bt \Ver_p)_{\KK}\to(\Ver_p)_{\KK},$$
where the last functor is just the tensor product. This is a $\iota$-tensor functor for the composite inclusion $\iota: \kk\subset\kk\subset\KK$, where the first inclusion is the Frobenius homomorphism.
\end{proof}

\begin{theorem}\label{ThmAllFields}
Let $\C$ be a Frobenius exact pre-Tannakian category of moderate growth. There exists a fiber functor $\C\to(\Ver_p)_{\KK}$ for some field extension $\kk\subset\KK$.
\end{theorem}
\begin{proof}
Let $\C$ be a Frobenius exact pre-Tannakian category of moderate growth. By Theorem~\ref{ThmPerf}(2), there exists a $\iota$-tensor functor $\C\to\C'$, for some field extension $\iota:\kk\hookrightarrow\KK$ to some Frobenius exact pre-Tannakian category over $\KK$ of moderate growth equipped with a tensor functor $I:(\Ver_p)_{\KK}\to\C'$ which makes $(\C',I)$ perfect (in particular reduced). By Proposition~\ref{PropRed}, the latter admits a fiber functor $\C'\to(\Ver_p)_{\KK'}$ for some field extension $\KK\subset\KK'$. 
\end{proof}

\subsection{Algebraically closed fields}\label{acf} 
Now we prove Theorem~\ref{MainThm}.
\begin{theorem}\label{MainThmText}
Assume that $\kk$ is algebraically closed and let $\C$ be a Frobenius exact pre-Tannakian category of moderate growth. There exists a fiber functor $\C\to\Ver_p$, which is unique up to isomorphism.
\end{theorem}
\begin{proof}
By an argument due to Deligne, see \cite[\S 6.4]{Co}, since $\kk$ is algebraically closed, any Tannakian category has a fiber functor to the category of vector spaces over $\kk$.
Therefore, we can take $\KK=\kk$ in Proposition~\ref{PropGradTan} (in this case, by construction, $\cD$ is actually the de-equivariantization of $\C$ by the affine group scheme $G$ for which $\C=\Rep G$) and Corollary~\ref{graded str cor}. Furthermore, for algebraically closed fields, we can interpret the $\Frob$-functor $\C\to\C^*\bt \Ver_p^{en}$ in the proof of Proposition~\ref{PropRed} as a $\kk$-linear tensor functor to a tensor category by twisting the action on the target, see Section~\ref{Frob1}. Finally, $\kk\hookrightarrow\KK$ from Theorem~\ref{ThmPerf}(2) is also an isomorphism for algebraically closed fields, so as in the proof of Theorem~\ref{ThmAllFields} we actually get a fiber functor $\C\to\Ver_p$.

Also by an argument due to Deligne, see \cite[\S 6.4]{Co}, since $\kk$ is algebraically closed, any two fiber functors from a pre-Tannakian category to $\sVec$ are isomorphic. So consider any fiber functor $F:\C\to\Ver_p$. By Example~\ref{ExFrFun}(i), $\C^\ast$ is a super Tannakian category and $\Fr:\C\to\C\boxtimes\Ver_p$ actually takes values in $\C^\ast\boxtimes \Ver_p$. Denote by $F_0$ the restriction of $F$ to a fiber functor $\C^\ast\to\sVec$. As an instance of the commutative diagram in Section~\ref{SecFrComm}, we get a commutative diagram
$$\xymatrix{
\C\ar[rr]^{F}\ar[d]^{\Fr}&&\Ver_p\ar[d]^{\Fr}\\
\C^\ast\boxtimes\Ver_p\ar[rr]^{F_0\boxtimes\id} &&\sVec\boxtimes\Ver_p\ar[r]^-{\otimes}&\Ver_p.
}$$
The upper path composes to the functor
$\Phi\circ F$, where $\Phi$ is the additive equivalence of $\Ver_p$ which sends every object to itself and sends every morphism to its $p$-th power. On the other hand, by uniqueness of fiber functors to $\sVec$, the lower path is unique up to isomorphism. It follows that also $F$ is unique.
\end{proof}

\section{Applications}

For the entire section, we assume that $\kk$ is algebraically closed, of characteristic $p\ge 0$.

\subsection{Frobenius-Perron dimensions in pre-Tannakian categories of moderate growth} 

\subsubsection{} We can apply Theorem \ref{MainThmText} to define 
Frobenius-Perron dimensions in Frobenius exact (in particular, semisimple) pre-Tannakian categories of moderate growth which are not necessarily finite. We agree that ${\rm Ver}_0:=\sVec_{\kk}$.

We let $\cO_p$ denote the ring $\BZ[2\cos(\pi/{p})]$ if $p>0$ and $\cO_0=\BZ$. Note that $\cO_p$ equals $\BZ$ if and only if $p\le 3$. Set also $q=e^{\pi i/p}$ if $p>0$. If $p>2$ then a $\BZ$-basis of $\cO_p$ is given by

$$[m]_q:=\frac{q^m-q^{-m}}{q-q^{-1}}, \quad\mbox{with $1\le m\le \frac{p-1}{2}$}.$$

\begin{theorem}\label{ThmFP}
Let $\C$ be a pre-Tannakian category, Frobenius exact if $p>0$, of moderate growth.
\begin{enumerate}
\item The assignment $X\mapsto \gd(X)$ induces a ring homomorphism
$$\mathrm{Gr}(\C)\,\to\,\cO_p,\quad [X]\mapsto \gd(X)$$
and $\gd(X)\ge [\dim X]_q$, for $\dim X\in\mathbb{F}_p\simeq\{0,1,\cdots, p-1\}$.
Moreover, if $p>2$, $\gd(X)$ is a linear combination with non-negative integer coefficients of the $[m]_q$ with $1\le m\le \frac{p-1}{2}$.
\item If $\C$ is a finite tensor category, then $\gd(X)=\FPdim(X)$ is the Frobenius-Perron dimension of $X$ (\cite{EGNO}, Section 4.5).
\item For the fiber functor $F:\C\to\Ver_p$ (which exists and is unique), we have $\gd(X)=\FPdim(F(X))$.
\item For a symmetric tensor functor $H:\C\to\C'$, where $\C'$ satisfies the same conditions imposed on $\C$, we have $\gd(X)=\gd(H(X))$.
\item $\C$ is Tannakian if and only if $\gd(X)=\ad(X)$ for all $X\in\C$ if and only if $\ad$ is multiplicative ($\ad(X\otimes Y)=\ad(X)\ad(Y)$ for all $X,Y\in\C$).
\item $\C$ is super-Tannakian if and only if $\gd(X)\in\BZ$ (equivalently $\gd(X)\in\BN$) for all $X\in\C$.
\end{enumerate}
\end{theorem}

Before proving Theorem \ref{ThmFP}, we need a couple of auxiliary results.

\begin{lemma}\label{LemVinX}
Consider a pre-Tannakian category $\C$, an object $X\in\C$ and a simple constituent $V$ of $X^{\otimes n}$ for $n\in\BN$. Then $V^\ast \otimes V$ is a subquotient of $S^n(X^\ast\otimes X)$.
\end{lemma}

\begin{proof}
If $p=0$, we can use the decomposition
$$S^n(X^\ast\otimes X)\;\simeq\;\bigoplus_{\lambda\vdash n}\mathbb{S}_\lambda X^\ast\otimes \mathbb{S}_\lambda X,$$
where $\mathbb{S}_\lambda$ is the Schur functor corresponding to the partition $\lambda$. In general, we have the following argument. Following \cite[\S 7]{De},  we consider the commutative Hopf algebra in $\Ind\C$
$$H\;=\; \int^{Z\in\C} Z^\ast\otimes Z\; =\;\cO(\pi_1(\C)).$$
Since this is the coordinate algebra of the fundamental group $\pi_1(\C)$ of $\C$, every object in $\C$ has a natural co-action of $H$, compatible with the tensor product. For each $Y\in\C$, we consider the corresponding (non-zero when $Y\not=0$) co-algebra morphism $\phi_Y:Y^\ast\otimes Y\to H$, obtained by adjunction from the co-action (or directly from the definition of $H$). By commutativity of $H$ we have an epimorphism
$$S^n(X^\ast\otimes X)\twoheadrightarrow \mathrm{im} \phi_{X^{\otimes n}}\subset H.$$
Since $V$ is a subquotient of $X^{\otimes n}$, we have $\mathrm{im}\phi_V\subset \mathrm{im}\phi_{X^{\otimes n}}$. The claim now follows from the observation that $\phi_V$ is a monomorphism for simple $V$. Indeed, $\phi_V$ is the image under the (exact) tensor functor $\C\boxtimes\C\to\C$ of a morphism
$$V^\ast\boxtimes V\to  \int^{Z\in\C} Z^\ast\boxtimes Z$$ in $\C\boxtimes\C$, where the left object is simple.
\end{proof}

\subsubsection{}\label{UnionFinite}
We let $\cD$ be a pre-Tannakian category which is a union of pre-Tannakian categories with finitely many isomorphism classes of simple objects. In other words, for every $X\in\cD$, there are only finitely many simple objects which appear as constituents in $\{X^{\otimes i}\otimes X^{\ast j}\,|\, i,j\in\BN\}$. By \cite[Proposition~4.5.7]{EGNO}, we have a well-defined notion of Frobenius-Perron dimension of objects $X$ in $\cD$, which is the Frobenius-Perron dimension of $X$ in an arbitrary tensor subcategory with finitely many simple objects which contains $X$.

Of immediate interest is the very specific case of the fusion category $\cD=\Ver_p$, but we choose this level of generality to be able to include $$\Ver_{p^\infty}\;=\;\bigcup_{n>0}\Ver_{p^n}$$
from \cite{BEO} for future applications.

\begin{lemma}\label{LemFP}
For all $X\in\cD$, we have
$$\FPdim(X)\,=\,\gd(X).$$
\end{lemma}

\begin{proof}
Let $d\in\BR$ be the maximal value the Frobenius-Perron dimension takes on a simple object in a tensor category $\C\subset\cD$ which contains $X$ and only has finitely many simple objects. Since the Frobenius-Perron dimension of a simple object is at least $1$, we find that for $n\in\BN$
$$\FPdim(X)^n\;\le\;d\cdot \ell(X^{\otimes n})\;\le\;d\cdot \FPdim(X)^n.$$
Hence
$$d^{-\frac{1}{n}}\FPdim(X)\;\le\;\ell(X^{\otimes n})^{\frac{1}{n}}\;\le\;\FPdim(X),$$
and taking the limit yields the result.
\end{proof}

\begin{lemma}\label{LemClaim}
Consider a pre-Tannakian category $\C$ and a symmetric tensor functor $F:\C\to\cD$, with $\cD$ as in \ref{UnionFinite}. For every object $X\in\C$ and $\epsilon>0$, there exists $C_\epsilon \in\BR$ such that, for every $n\in\BN$ and simple constituent $V$ of $X^{\otimes n}$, we have $\FPdim (FV)\le C_\epsilon(1+\epsilon)^n$.
\end{lemma}
\begin{proof}
Since $ \FPdim (F(V^\ast \otimes V))$ is the square of $\FPdim (FV)$, via Lemma~\ref{LemVinX} it suffices to show that for every $Y\in\cD$, there exists $D_\epsilon\in\BR$ such that $\FPdim (S^nY)\le D_\epsilon (1+\epsilon)^n  $ for all $n\in\BN$.
The latter follows from \cite[Proposition~11.1]{EOf}.
\end{proof}

\begin{lemma}\label{PropFP1}
For a pre-Tannakian category $\C$ and a symmetric tensor functor $F:\C\to\cD$, with $\cD$ as in \ref{UnionFinite}, we have
$$\gd (X)\;=\;\gd(FX),\qquad\mbox{for all $X\in\C$.}$$

\end{lemma}
\begin{proof}
Clearly, we have $\gd(X)\le \gd(FX)$. On the other hand, we have
$$\ell(FX^{\otimes n})=\sum_{V} [X^{\otimes n}:V] \ell(FV)\le \sum_{V} [X^{\otimes n}:V] \FPdim(FV),$$
where the sum runs over the simple objects $V$ which appear as constituents in $X^{\otimes n}$.
Hence, by Lemma~\ref{LemClaim}, for every $\epsilon>0$ there exists $C_\epsilon \in\BR$ 
$$\ell(FX^{\otimes n})^{\frac{1}{n}}\;\le\;   (1+\epsilon)\ell(X^{\otimes n})^{\frac{1}{n}} C_\epsilon^{\frac{1}{n}}.$$
Thus $\gd(FX)$ is bounded by $(1+\epsilon)\gd(X)$ for every $\epsilon>0$, so $\gd(FX)\le \gd(X)$.
\end{proof}

\begin{proof}[Proof of Theorem~\ref{ThmFP}]
Part (3) follows from the combination of Lemmas~\ref{LemFP} and~\ref{PropFP1}.

For $p>0$, recall from \cite{EOV} that the Frobenius-Perron dimension yields a ring homomorphism
$$\FPdim:\;\mathrm{Gr}(\Ver_p)\;\twoheadrightarrow\; \BZ[2\cos(\pi/{p})],   \quad [L_r]\mapsto [r]_{q}.$$
Note that the function $\BZ\to \BR$, $n\mapsto |[n]_q|$ is subadditive, since $|\sin (x)|$ is subadditive on $\BR$. The function factors through the ring homomorphism $\BZ\twoheadrightarrow\mathbb{F}_p$ which thus implies that the function
$$[\cdot]_q:\mathbb{F}_p\to\BR_{\ge 0},\;\, n\mapsto [n]_q=|[n]_q|$$
satisfies $[i+j]_q\le [i]_q+[j]_q$ for all $i,j\in\mathbb{F}_p$ (with the sum $i+j$ taken in $\mathbb{F}_p$). Since $\FPdim(L)=[\dim L]_q$ for each simple object $L\in\Ver_p$, it follows that $\FPdim(X)\ge[\dim X]_q$ for all $X\in\Ver_p$.

By the above paragraph, part (3) implies part (1). Since the Frobenius-Perron dimension is preserved by tensor functors between finite tensor categories, part (3) also implies part (2). Part (4) similarly follows from part (3) by considering a fiber functor $\C'\to\Ver_p$.

Part (6) is only relevant if $p>2$. By the above description of the Frobenius-Perron dimension on $\Ver_p$, part (3) shows that a fiber functor $F:\C\to\Ver_p$ takes values in the category of supervector spaces if and only if $\gd$ takes values in the integers. By part (6) and part (3), to prove the first characterisation in part (5) it suffices to observe that for a supervector space $W$, we have $\gd(W)=\ad(W)$ if and only if $W$ is concentrated in degree $0$, see Example~\ref{ExVect}.

The second characterisation in part (5) is trivial $p=2$. We observe that 
$$1=\ad(L_{p-1}^{\otimes 2})\;<\; (\ad L_{p-1})^2=(p-1)^2,\quad\mbox{if $p>2$},$$
$$4=\ad(L_{p-2}^{\otimes 2})\;<\; (\ad L_{p-2})^2=(p-2)^2,\quad\mbox{if $p>3$}.$$
Lemma~\ref{LemXY} thus shows that if $\ad(X^{\otimes n})=\ad(X)^n$ for all $n$, for some $X\in \Ver_p$, then $X\in\Ve$, from which the claim follows.
\end{proof}

\begin{conjecture}\label{semiconj} In any pre-Tannakian category 
of moderate growth the map $V\mapsto \gd(V)$ defines a homomorphism ${\rm Gr}(\mathcal C)\to \Bbb R$, and $\gd(V)$ is an algebraic integer.  
\end{conjecture} 

\begin{remark} \begin{enumerate}
\item Theorem \ref{ThmFP} shows that Conjecture \ref{semiconj} holds in characteristic zero and for Frobenius exact categories in positive characteristic. 

\item Conjecture \ref{semiconj} follows from Lemmas~\ref{LemFP} and~\ref{PropFP1} together with \cite[Conjecture~1.4]{BEO}. 
Moreover, \cite[Conjecture 1.4]{BEO}  constrains the precise form of algebraic integers 
$\gd(V)$, which turn out to be somewhat more general than 
in the Frobenius exact case (namely, they are Frobenius-Perron dimensions of objects in $\Ver_{p^n}$ rather than $\Ver_p$).
\end{enumerate}
\end{remark}  

\subsubsection{} Let $\CC$ be a Frobenius exact pre-Tannakian category of moderate growth over $\kk$, ${\rm char}(\kk)=p>0$, and let $F: \C\to \Ver_p$ be its fiber functor. Denote by $\CC_{\rm int}\subset \CC$ the full subcategory of objects with integer $\gd(X)$. Then by Theorem \ref{ThmFP}, $\CC_{\rm int}\subset \CC$ is a super-Tannakian tensor subcategory which coincides with $\C$ if $p=2,3$. 

\begin{proposition}\label{coinci} We have $\C_{\rm int}=\C^*$. 
\end{proposition}

\begin{proof} For $p=2,3$ both categories coincide with $\C$, so we may assume $p>3$. Suppose $X\in \C_{\rm int}$. Then by Theorem \ref{ThmFP} $\gd({\rm Fr}(X))=\gd(X)$ is an integer, so 
we must have ${\rm Fr}(X)\in \C\boxtimes \sVec$, thus $X\in \C^*$. Conversely, if $X\in \C^*$ then 
${\rm Fr}(X)\in \C\boxtimes \sVec$, so ${\rm Fr}(F(X))\in \Ver_p\boxtimes \sVec$. Hence 
$F(X)\in \sVec$, so by Theorem \ref{ThmFP} $\gd(X)=\gd(F(X))\in \Bbb Z$ and $X\in \C_{\rm int}$. 
\end{proof}

\subsubsection{} Assume now that $\C$ is finitely generated. 
We will represent $\C$ as the equivariantization of a finite symmetric tensor category 
with respect to the action of an affine group scheme of finite type over $\bold k$. If $p=2$ then we have shown in Corollary \ref{peq2} that $\CC$ is Tannakian, i.e.,  $\CC=\Rep G={\rm Vec}_\kk^G$ for an affine group scheme $G$ of finite type over $\kk$, so we may assume that $p>2$. The fiber functor $\C\to\Ver_p$ allows us to define a homomorphism $\varepsilon:\pi_1(\Ver_p)\to \mathbb G$ between affine group schemes of finite type in $\Ver_p$ (we refer to \cite{Ve} for details on the theory of such affine group schemes), such that $\C\simeq\Rep(\mathbb G,\varepsilon)$. Then we define $G$ as the maximal ordinary subgroup of $\mathbb G$. Its coordinate ring is defined from $\mathcal{O}(\mathbb G)$ by taking the quotient with respect to the ideal generated by the maximal subobjects which is a direct sums of objects $L_i,1<i$ in $\Ver_p$. Alternatively, $G$ is the restriction of the functor $\mathbb G$ to the category of ordinary algebras.

We then have a homomorphism $\mathbb G\to G^{(1)}$ to the Frobenius twist of $G$, which comes from 
$$\Fr_+\cO(\mathbb G)\hookrightarrow S^p\cO(\mathbb G)\to\cO(\mathbb G),$$
where the second morphism is multiplication. Indeed, we have $\Fr_+\cO(G)\simeq\cO(G)^{(1)}$ and, using that the $p$-th symmetric power of every $L_i,i>1$ in $\Ver_p$ vanishes, it follows that $\Fr_+J$ maps to zero in the above morphism, with $J$ the ideal in $\mathcal{O}(\mathbb G)$ defining $\mathcal{O}(G)$.
We denote the kernel and image of $\mathbb G\to G^{(1)}$ by $\mathbb G_1$ and $G_+$ (in particular, $G_+= G^{(1)}$ if $G$ is reduced). It follows that ${G_+}$ is an affine group scheme of finite type, that $\mathbb{G}_1$ is a finite group scheme in $\Ver_p$ and that $\varepsilon$ factors through $\mathbb{G}_1$. The following proposition then follows
from standard properties of equivariantization and de-equivariantization, see \cite{EGNO}.\footnote{To be more precise, the book \cite{EGNO} considers (de)equivariantization only for abstract group actions, but the theory for actions of affine group schemes is completely parallel, with the usual algebro-geometric amendments, see e.g. \cite{G}.}

\begin{proposition}\label{PropFinG1}
The tensor category $\Rep(\mathbb{G}_1,\varepsilon)$ is finite and for the canonical action of ${G_+}$, we have $\C\simeq (\Rep(\mathbb{G}_1,\varepsilon))^{G_+}$.
\end{proposition}

\begin{remark} 
 This kind of presentation does not always exist in characteristic zero, 
e.g. for $\CC=\Rep SL(m|n)$. The reason it exists in characteristic $p$ is that 
the Frobenius homomorphism of $SL(m|n)$ lands in the even part $SL(m|n)_0=S(GL(m)\times GL(n))$, while in characteristic zero $SL(m|n)$ has no non-trivial homomorphisms to purely even groups when $m,n>0$. 
\end{remark}

\subsection{Classification of semisimple pre-Tannakian categories of moderate growth in positive characteristic} 
Using Nagata's classification of {\it linearly reductive} affine group schemes, i.e., those with semisimple representation categories (\cite{N}), we can give an  even more explicit description of {\it semisimple} pre-Tannakian categories of moderate growth in positive characteristic. 

Namely, let $\CC$ be such a category, and assume that $\CC$ is finitely generated (we may assume this without loss of generality since any such category is a union of finitely generated ones). By Proposition~\ref{PropFinG1}, 
$\CC=\mathcal D^{{G_+}}$ where $\mathcal D$ is a finite symmetric category 
with an action of a linearly reductive finite type group scheme ${G_+}$. 
Moreover, in this case it is easy to see that for $\C$ to be semisimple, $\mathcal D$ must be semisimple, i.e., a fusion category. Thus we get 

\begin{theorem} Let $\CC$ be a finitely generated semisimple pre-Tannakian category of moderate growth over $\kk$ of characteristic $p>0$. Then there exists a symmetric fusion category $\mathcal D$ and a linearly reductive affine group scheme ${G_+}$ of finite type acting on $\mathcal D$ such that $\CC=\mathcal D^{{G_+}}$. 
\end{theorem}

Thus we see that $\CC$ has a Tannakian subcategory $\Rep {G_+}$ over which it is a finite module category. Consequently, we have 

\begin{corollary}\label{corol} There exists a finite collection of simple objects 
$X_1,...,X_r\in \CC$ such that every simple object of $\CC$ 
is a direct summand in $V\otimes X_i$ for some $i$, where 
$V\in \Rep {G_+}$ is irreducible. 
\end{corollary} 

\begin{proof}
Let $H: \C\to\C_{{G_+}}$ be the natural surjective tensor functor.

Let $Y_1,...,Y_r$ be all the simples of $\cD=\C_{{G_+}}$ and $X_1,...,X_r$ be some simple objects of $\C$ such that $H(X_i)$ contains $Y_i$; they exist by surjectivity of $H$. So $\Hom(H(X_i),Y_i)\ne 0$. Let $H^\vee: \C_{{G_+}}\to {\rm Ind}\C$ be the right adjoint of $H$, i.e., the induction functor. Then $H^\vee(H(X_i))=H^\vee(\bold 1)\otimes X_i=\mathcal O({G_+})\otimes X_i$. Hence for an arbitrary $X\in \C$, we have 
$$\Hom(X,\mathcal O({G_+})\otimes X_i)=\Hom(X,H^\vee(H(X_i))=\Hom(H(X),H(X_i)),$$ which is nonzero for some $i$ if $X\ne 0$ as $H(X_i)$ contains $Y_i$. So for simple $X$ there is $i$ such that $X$ is contained in $\cO({G_+})\otimes X_i$. Hence there is a simple ${G_+}$-module $V$ such that $X$ is contained in $V\otimes X_i$, as claimed. 
\end{proof}

Moreover, Nagata's theorem (\cite{N}, see also Section \ref{SecBen} below) allows us to describe ${G_+}$ quite explicitly. Namely, such a group scheme can be included in a short exact sequence
$$
1\to A^\vee\to {G_+}\to \Gamma\to 1,
$$
where $\Gamma$ is a finite group of order prime to $p$, $A$  is a finitely generated abelian group without $p$-torsion, and $A^\vee$ is the dual group scheme of $A$
(i.e., $A^\vee=T\times A_{\rm tors}^\vee$, the product of a torus 
with the dual group scheme to a finite abelian $p$-group). 
It  follows that the dimensions of the irreducible representations of ${G_+}$ are uniformly bounded (by $|\Gamma|$). Thus we get 

\begin{corollary} Frobenius-Perron dimensions of simple objects in a finitely generated semisimple pre-Tannakian category of moderate growth in positive characteristic
are uniformly bounded (so they take finitely many values). 
\end{corollary}

\begin{remark} Note that this is false in characteristic zero, e.g. for 
$\CC=\Rep SL_2(\Bbb C)$. 
\end{remark} 

\subsection{Growth rates in modular representation theory}

Now we will discuss some applications of our results to classical modular representation theory. First we will consider the problem of describing non-negligible indecomposable direct summands in $V^{\otimes n}\otimes V^{*\otimes m}$ (i.e., ones of dimension prime to $p$), where $V$ is a finite dimensional representation of a finite group (or, more generally, affine group scheme) $\bold G$  over a field $\kk$ of characteristic $p>0$. This type of questions is discussed in \cite{B,B2,BS}.

\subsubsection{}
Let $\bold G$ be an affine group scheme over $\kk$ of characteristic $p>0$. Let $V\in \Rep\bold G$. Let $\delta_n(V)$ be the number of indecomposable non-negligible direct summands in $V^{\otimes n}$ (counted with multiplicities). We have $\delta_n(V)\le (\dim_\kk V)^n$. Thus 
we can define the invariant 
$$
\delta(V):=\limsup_{n\to \infty}\delta_n(V)^{1/n}\in \Bbb R. 
$$ 

\begin{theorem}\label{finti} 
\begin{enumerate}
\item $\delta(V)=\lim_{n\to \infty}\delta_n(V)^{1/n}=\sup_{n\ge 1}\delta_n(V)^{1/n}$. Moreover, there exist unique non-negative integers $m_j$, $j=1,...,p-1$ such that
$$
\delta(V)=\sum_{k=1}^{p-1}[k]_qm_k
$$
and for $p>2$
$$
\delta(S^2V)-\delta(\wedge^2V)=\sum_{k=1}^{p-1}[k]_{q^2}m_k.
$$

\item For $V,W\in \Rep\bold G$,
$$\delta(V\oplus W)=\delta(V)+\delta(W)\;\mbox{ and }\;\delta(V\otimes W)=\delta(V)\delta(W).$$

\item $\dim_\kk V-\sum_{k=1}^{p-1}km_k$ is divisible by $p$. 

\item We have $\delta(V)\ge |[\dim_{\kk}V]_q|.$

\end{enumerate}
\end{theorem}

Note that the bound in (4) is sharp (achieved for $\bold G=\Bbb Z/p$). 

\begin{example} For $p=2,3$ we get that $\delta(V)$ is an integer. For $p=5$ we have $\delta(V)=n_1+\frac{1+\sqrt{5}}{2}n_2$ for non-negative integers $n_1,n_2$.
\end{example}

\begin{proof} (1) It is clear 
that $\delta_n(V)=d_n(\overline V)$, where $\overline V$ is the image of $V$ in the semisimplification of $\Rep\bold G$. Thus 
$$
\delta(V)=\lim_{n\to \infty}\delta_n(V)^{1/n}=\sup_{n\ge 1}\delta_n(V)^{1/n}=\gd(\overline V).
$$
Let $\mathcal C=\langle \overline{V}\rangle$ and 
$F: \mathcal C\to {\rm Ver}_p$ be the fiber functor, which exists by Theorem \ref{MainThmText}. Then $F(\overline{V})=\oplus_{k=1}^{p-1}m_kL_k$ and ${\rm FPdim}(L_k)=[k]_q$, so
$$
{\rm FPdim}(\overline{V})=\sum_{k=1}^{p-1}[k]_q m_k
$$
and for $p>2$
$$
{\rm FPdim}(S^2\overline{V})-{\rm FPdim}(\wedge^2\overline V)=
\sum_{k=1}^{p-1}[k]_{q^2}m_k,
$$
see \cite{EOV}, Proposition 4.5.

Let $p>2$. Since $[k]_q=[p-k]_q$, and $[k]_q$ form a $\Bbb Q$-basis of
$\Bbb Q(q+q^{-1})$ when $1\le k\le \frac{p-1}{2}$,
we see that $\delta(V)$ determines $m_k+m_{p-k}$.
On the other hand, applying the Galois automorphism $g$
such that $g(q^2)=-q$, we get
$$
g({\rm FPdim}(S^2\overline{V})-{\rm FPdim}(\wedge^2\overline V))=
\sum_{k=1}^{p-1}(-1)^{k-1}[k]_{q}m_k,
$$
which determines $m_k-m_{p-k}$. Thus $m_k$ are uniquely determined, as claimed.

(2) This follows immediately from the above and Theorem~\ref{ThmFP}(1).

(3) The image of $\dim_\kk V$ in $\Bbb F_p$
is $\dim \overline V=\sum_{k=1}^{p-1}km_k\in \Bbb F_p$.
Thus the number $\dim_\kk V-\sum_{k=1}^{p-1}km_k$ is divisible by $p$.

(4) This follows from Theorem~\ref{ThmFP}(1), using $|[\dim_k V]_{q}|=[\dim V]_q$, for the categorical dimension $\dim V=\dim \overline{V}$.
\end{proof}

Also, Corollary \ref{corol} and Nagata's theorem imply

\begin{corollary} Let $V$ be an indecomposable representation of $\bold G$ over $\kk$ of dimension prime to $p$, and let $\Sigma(V)$ be the collection of all indecomposables of dimension prime to $p$ that occur as direct summands in $V^{\otimes n}\otimes V^{*\otimes m}$ for various $n,m$. 
Then there exists a constant $K_V$ such that for any $W_1,...,W_r\in \Sigma(V)$, the tensor product $W_1\otimes...\otimes W_r$ contains $\le K_V^r$ indecomposable direct summands of dimension coprime to $p$. In particular, for every $W\in \Sigma(V)$ one has $\delta(W)\le K_V$. 
\end{corollary} 

\begin{remark}
The growth rate $\delta(V)$ is an analog of $\gamma_{\mathfrak{X}}(V)$ introduced in \cite[\S 1.4]{B} ($\mathfrak{X}$ is the ideal of negligible objects in $\Rep \mathbf{G}$), which is similarly defined in terms of the {\em dimension} of the non-negligible part of $V^{\otimes n}$. In particular, $\delta(V)\le \gamma_{\mathfrak{X}}(V)$. To the best of our knowledge, it is an open question whether $\gamma_{\mathfrak{X}}$ takes values in algebraic numbers (as $\delta$ does by \ref{finti}(1)), whether the properties in \ref{finti}(2) hold for $\gamma_{\mathfrak{X}}$, or whether in fact $\delta=\gamma_{\mathfrak{X}}$.
\end{remark}

\subsection{Benson's conjecture}\label{SecBen}

\subsubsection{} We start by formulating the infinite type version of Nagata's theorem. Recall that for any abelian group, the dual group scheme $A^\vee$ from~\cite[IV.\S1]{DG} can be defined as $\mathrm{Spec}( \kk A)$. A profinite group $\varprojlim \Gamma_i$ can be interpreted as an affine group scheme by taking the inverse limit in the corresponding category, or explicitly as $\mathrm{Spec}\varinjlim (\kk\Gamma_i)^\ast$.
\begin{lemma}\label{LemNag}
Let $G$ be an affine group scheme over $\kk$ for which $\Rep G$ is semisimple. Then there exists a unique (up to isomorphism) short exact sequence
$$1\to A^\vee\to G\to \Gamma\to 1,$$
where $A$ is an abelian group in which all torsion elements have order a power of $p$ and $\Gamma$ is a projective limit of finite groups with order prime to $p$.
\end{lemma}

\begin{proof}
Uniqueness follows from the observation that for such groups, all homomorphisms $A^\vee\to\Gamma$ are trivial. From the Tannakian formalism it follows that $G$ must be an inverse limit of affine group schemes $H$ of finite type with semisimple representation theory. As $\kk$ is algebraically closed and by \cite[IV.3.1.1.2]{DG}, Nagata's theorem \cite[IV.3.3.6]{DG} implies that the short exact sequence given by the inclusion of the connected component $H^0$ in $H$ is a short exact sequence of the desired type. By the above vanishing of homomorphisms, the short exact sequences align well to take the inverse limit.
\end{proof}

\subsubsection{}\label{DefFun}
Consider now the case of characteristic $p=2$. In this case we know by Corollary~\ref{peq2} that for a finite group $\bold G$ the semisimplification $\overline{\Rep \bold G}$ is Tannakian, i.e., $\overline{\Rep \bold G}=\Rep G$ for an affine group scheme $G$ as in Lemma~\ref{LemNag}. In particular, we have an assignment 
$${\bold G}\;\mapsto\;(A=A_{\bold G},\Gamma=\Gamma_{\bold G}),$$
which associates to any finite group $\bold G$ an abelian group $A$ without odd torsion and a profinite group $\Gamma$ which is the inverse limit of finite groups of odd order. Note that we then actually have
$$\overline{\Rep\bold G}\;\simeq\; (\Ve_A)^{\Gamma}$$
for an action of $\Gamma$ on the category of $A$-graded vector spaces.

This is closely related to the following conjecture of D. Benson (\cite{B2}, Conjecture 1.1), supported by ample computer evidence and proofs in special cases.

\begin{conjecture}\label{beco}
  Let $\bold G$ be a finite $2$-group and $V$ an odd dimensional indecomposable representation of $\bold G$ over $\kk$ of characteristic $2$ and let us decompose $V \otimes V^{*}$ as the direct sum $\kk \oplus Q$. Then, $Q$ is a direct sum of even-dimensional indecomposable representations.\footnote{In fact, it is conjectured in \cite{B2} that the dimensions of the non-trivial summands in $V\otimes V^*$ are moreover divisible by $4$.}
\end{conjecture}

In the language of tensor categories, Conjecture \ref{beco} says that the semisimplification $\overline{{\rm Rep}\bold G}$ is a {\em pointed category}, i.e., the pre-Tannakian category ${\rm Vec}_{A'}$ of vector spaces graded by an abelian group $A'$, and in particular for any indecomposable odd-dimensional $V\in \Rep \bold G$, the subcategory
$\langle \overline V\rangle$ is ${\rm Vec}_C$ for a cyclic group $C$.

Moreover, computer evidence collected by D. Benson (private communication) suggests that $C=\Bbb Z$ or has order a power of $2$. Such strengthened version of the conjecture would be equivalent to

\begin{conjecture}\label{gammatriv} If $\bold G$ is a finite 2-group,
then $\Gamma_{\bold G}=1$. In other words,
 $\overline{\Rep\bold G}={\rm Vec}_A$ for some (in general, infinitely generated) abelian group $A$ without odd torsion.
\end{conjecture}

For a general finite group $\bold G$, it is known (\cite{EOf}, Section 4) that in any characteristic $p$
$$
\overline{{\rm Rep}\bold G}\simeq \overline{{\rm Rep} N(\bold G_p)}\simeq \overline{{\rm Rep} \bold G_p}^{N(\bold G_p)/\bold G_p},
$$
where $\bold G_p$ is a Sylow $p$-subgroup of $\bold G$, $N(\bold G_p)$ its normalizer, and the superscript $N(\bold G_p)/\bold G_p$ means taking the $N(\bold G_p)/\bold G_p$-equivariantization.
Returning to $p=2$, we thus find a short exact sequence
$$
1\to \Gamma_{\bold{G}_2}\to \Gamma_{\bold G}\to N(\bold G_2)/\bold G_2\to 1.
$$
So Conjecture \ref{gammatriv} is equivalent to the following seemingly more general conjecture.

\begin{conjecture}\label{gammatriv1} For any finite group $\bold G$, the natural surjective homomorphism $\Gamma_{\bold G}\twoheadrightarrow N(\bold G_2)/\bold G_2$ is an isomorphism $\Gamma_{\bold G}\simeq N(\bold G_2)/\bold G_2$.
\end{conjecture}



Finally, note that by the Feit-Thompson theorem, a finite group of odd order is solvable. Hence if $\Gamma_{\bold G}\ne 1$ then it has a non-trivial algebraic 1-dimensional representation. 
So Benson's Conjecture \ref{gammatriv} is equivalent to the following conjecture. 

\begin{conjecture} Let $\bold G$ be a finite 2-group, $V$ a non-trivial indecomposable representation of $\bold G$ over $\bold k$ 
of characteristic $2$, and $m$ an odd positive integer. Then 
$V^{\otimes m}$ does not contain the trivial representation $\kk$ as a direct summand. 
\end{conjecture}

We don't know if this is true even for $m=3$.


The function ${\bold G}\mapsto (A,\Gamma)$ from \ref{DefFun} is only known in a few cases, in most cases $A$ is too big to describe explicitly (whereas for tame representation type, $A$ is countable, already for $(\BZ/2)^{\times 3}$ it has cardinality of $\kk$). Since they all confirm the conjectures in this section, we restrict to mentioning $A$ for 2-groups.
\begin{example}(\cite{A}, \cite[\S 3]{B2})
Let $n$ be a strictly positive integer and $q=2^n$.
\begin{center}
\begin{tabular}{|c|c|}
\hline
${\bold G}$& $A$\\
\hline
\hline
$\BZ/2^n$&$(\BZ/2)^{\times n-1} $\\
\hline
$(\BZ/2)^{\times 2}$&$\BZ $\\
\hline
$D_{4q}$&\mbox{infintely generated torsion free}\\
\hline
\end{tabular}
\end{center}
For the generalized quaternion groups $Q_{4q}$ (including $Q_8$), it is known there is $4$-torsion in $A$, but it is not known that $\Gamma=1$.
\end{example}

\subsection{Characteristic $p>2$} In this subsection we prove a generalization 
of Theorem 1.4 in \cite{B2}. This is only one example of many possible applications 
of this type, which we give here as a ``proof of concept".

We assume that $\bold k$ is an algebraically closed field 
of characteristic $p\ge 3$ and ${\bold G}$ a finite group scheme over $\bold k$. We 
will freely use the terminology of \cite{B2}. 

Recall  that to every finite-dimensional representation $V$ of ${\bold G}$, Friedlander, Pevtsova and Suslin \cite{FPS}
attached a {\it generic Jordan type} $[1]^{m_1}...[p]^{m_{p}}$ and the corresponding {\it stable generic Jordan type} 
$[1]^{m_1}...[p-1]^{m_{p-1}}$ (see also \cite{B2} and references therein). Namely, we have a variety of $\pi$-points $\Pi({\bold G})$ as introduced in \cite[\S 2]{FPS}, which comprises equivalence classes of homomorphisms $\alpha: \kk[x]/x^p\to \kk{\bold G}$ factoring through the group algebra of an abelian subgroup scheme of ${\bold G}$. For a fixed choice of irreducible component $Z\subset\Pi(\bold G)$, there is a unique generic Jordan type $[1]^{m_1}...[p]^{m_{p}}$ of $\alpha^*V$, where $m_i$ denotes the number of copies of the indecomposable of length $i$ in a decomposition into indecomposable $\kk[x]/x^p$-modules of $\alpha^*V$, for each $V\in\Rep\bold G$, see \cite[Theorem~3.4]{FPS}.


\begin{lemma}\label{le1} One has $\delta(V)\le \sum_{k=1}^{p-1} [k]_qm_k$. 
\end{lemma} 

\begin{proof} Let $X:=\oplus_{k=1}^{p-1} m_kL_k\in \Ver_p$.
It is known that the tensor product for stable generic Jordan types is given by the Verlinde rule, i.e., is the same as in ${\rm Ver}_p$ (\cite{FPS}, \cite{B2}, Section 4). 
Thus the number $s_n(V)$ of summands in the stable generic Jordan type 
of $V^{\otimes n}$ equals the length of the object 
$X^{\otimes n}$. Since the dimension of a non-negligible indecomposable object in $\Rep\bold G$ is not divisible by $p$, its stable generic Jordan type is not empty.
Applying this to all indecomposable non-negligible direct summand in $V^{\otimes n}$ shows
$$
\delta_n(V)\le s_n(V)=\ell(X^{\otimes n}). 
$$
Taking $n$-th root and limit $n\to \infty$, we obtain 
$$
\delta(V)\le {\rm FPdim}(X)=\sum_{k=1}^{p-1}[k]_qm_k. 
$$   
\end{proof}

Following \cite{B2}, we will say that a finite dimensional representation $V$ of a finite group ${\bold G}$ is {\it $p'$-invertible} if $V\otimes V^*\cong \bold 1\oplus P$ where $P$ is negligible (i.e., a direct sum of indecomposable modules of dimension divisible by $p$).
Recall that by \cite[Lemma~4.2]{B2}, if a stable generic Jordan type of $V$ is $[1]$ or $[p-1]$ then $V$ is $p'$-invertible. So consider the next simplest case -- stable generic Jordan types
$[2], [p-2], [1]^2, [p-1]^2$. 

\begin{theorem} Let $M$ be an indecomposable finite dimensional representation of a finite group ${\bold G}$ over $\bold k$ which is not $p'$-invertible. Suppose that a stable generic Jordan type of $M$ is 
$[2]$, $[p-2]$, $[1]^2$, or $[p-1]^2$. Then $S^2M$ (for $[p-2],[p-1]^2$) or $\wedge^2M$ (for $[2],[1]^2$) is $p'$-invertible. Moreover, the McKay graph $Q_M$ is a quotient of a graph of the form $\Bbb ZT$ by an admissible group of automorphisms, where $T$ is a directed labelled tree. The associated undirected labelled tree $\overline T$ is uniquely determined by $Q_M$ and is one of the following.

(i) For Jordan type $[2],[p-2]$ (in which case $p\ge 5$): Dynkin diagram $A_{p-1}$. 

(ii) For Jordan type $[1]^2,[p-1]^2$: Dynkin diagram $A_{p-1}$, affine Dynkin diagram $\widetilde D_n$, $n\ge 4$, $\widetilde E_6$, $\widetilde E_7$, $\widetilde E_8$, or the infinite diagram $D_\infty$. Moreover, $A_{p-1}$, $\widetilde E_6$, $\widetilde E_7$, $\widetilde E_8$ do not occur for $p=3$ and $\widetilde E_8$ does not occur for $p=5$.  
\end{theorem} 

\begin{proof} By Lemma \ref{le1}, we have $\delta(M)\le [2]_q$ in case (i) and $\delta(M)\le 2$ in case (ii). Since we also have $\delta(M)>1$ (as $M$ is not $p'$-invertible), by Theorem \ref{finti} we have 
$\delta(M)=[2]_q$ in case (i) and $\delta(M)=[2]_q$ (with $p>3$) or $\delta(M)=2$ in case (ii). 
This implies that $S^2M$ or $\wedge^2M$ (whichever has dimension $1$ mod $p$) is $p'$-invertible. 

Let $\overline M$ be the image of $M$ 
in the semisimplification $\overline{\Rep {\bold G}}$; then $S^2\overline M$ or $\wedge^2\overline M$ 
is invertible. Consider the tensor subcategory $\C_0\subset \overline{\Rep {\bold G}}$ generated by $\overline M$  
and define the product $\C:=\C_0\boxtimes \sVec$. Let $X\in \C$ be $\overline M$ if $\wedge^2\overline M$ is invertible, otherwise $X=\overline M\otimes \psi$, where 
$\psi\in \sVec$ is the generating invertible object. Then $\wedge^2X$ is invertible. 

If $\delta(M)=[2]_q$ (and $p\ge 5$), the result follows from\footnote{We only use Theorem 4.79 in the simple special case of semisimple categories.} \cite[Theorem 4.79]{BEO},  applied to the subcategory $\C_X\subset 
\C$ generated by $X$, and the tree $\overline{T}$ must be the Dynkin diagram $A_{p-1}$. 

On the other hand, if $\delta(M)=2$, then \cite[Theorem~4.79]{BEO} tells us that $\C_X=\Rep(K,z)$ for a linearly reductive subgroup scheme $K\subset GL_2(\bold k)$, such that the tautological representation $\bold k^2$ of $K$ is irreducible. So the result follows from the well known classification of such subgroup schemes. Namely, the undirected graph $\overline{T}$ is just the McKay graph of $L:=K\cap SL_2(\bold k)$, and classification of such subgroup schemes $L\subset SL_2(\bold k)$ 
is essentially the same as in characteristic zero (given by McKay's correspondence with simply laced affine Dynkin diagrams, including infinite ones), except that types $\widetilde A_n$ and $A_\infty^\infty$ (corresponding to abelian subgroup schemes) are excluded due to irreducibility of $\bold k^2$ and $A_\infty$ (corresponding to the entire $SL_2(\bold k)$) due to linear reductivity. Also we see that $\widetilde E_m$ do not occur in characteristic $3$ and $\widetilde E_8$ in characteristic $5$ (as the corresponding Kleinian groups have elements of order $3$ and $5$, respectively) and $A_2$ does not occur in characteristic $3$. This completes the proof. 
\end{proof} 

In particular, in characteristic 3 we obtain a strengthening of \cite{B2}, Theorem 1.4, with a number of diagrams ruled out; namely, the only possible diagrams are $\widetilde D_n, n\ge 4$ and $D_\infty$. 

We note that all the listed diagrams do occur (except possibly $A_{p-1}$ in case (ii)) if 
we take arbitrary (not necessarily unipotent) group schemes ${\bold G}$; e.g., we can simply take ${\bold G}=\Bbb Z/p$
and $V=\bold k^2$ in case $1$, and ${\bold G}=K$ in case (ii). However, it is not obvious which of them occur 
for unipotent group schemes.

\subsection{$p$-adic dimensions}
Assume $p>0$.
\begin{theorem}
Let $\C$ be a symmetric tensor category.
\begin{enumerate}
\item For $X\in\C$, there exists $\Dim_aX\in\BZ_p$, with $p$-adic expansion 
$$
\Dim_aX=\sum_{k\ge 0}t_kp^k,
$$ 
such that in $\mathbb{F}_p[[z]]$ we have 
$$\sum_{j\ge 0}\dim (\mathrm{A}^jX)z^j\;=\; \prod_{k\ge 0} (1+z^{p^k})^{t_k}.$$
\item If $\C$ is of moderate growth, then $\Dim_aX\in\BN\subset\BZ_p$ and $\Dim_aX\le \ad(X)$.
\item If $\C$ is of moderate growth and Frobenius exact, then $\Dim_aX= \ad(X)$.
\end{enumerate}
\end{theorem}
\begin{proof}
For part (1), the proof of \cite[Theorem~2.3]{EHO} carries over almost verbatim, it suffices to update the proofs of Claims 1 and 2 in \cite[2.2]{EHO}, which is left as an exercise. By Corollary~\ref{CorChain}, we know $\ad(X)<\infty$, so $\mathrm{A}^jX=0$ if $j>\ad(X)$. Part (2) then follows. By Theorem~\ref{MainThmText} it suffices to prove part (3) for $\Ver_p$, which follows from Example~\ref{ExDima} below.
\end{proof}

\begin{example}\label{ExDima}
\begin{enumerate}
\item By Example~\ref{ExVerpad}, in $\Ver_p$ we have $\Dim_aL_i=i=\ad(L_i)$ for all $0<i<p$.
\item For the projective object $P$ in $\Ver_4^+$, we have $\Dim_aP=4=\ad(P)$.
\end{enumerate}
\end{example}

\appendix

\section{A faithfulness result for symmetric groups }
\begin{center}
\textbf{by Alexander Kleshchev}
\end{center}
\vspace{2mm}

In the entire section, we let $\kk$ be a field of characteristic $p>0$.

\begin{proposition}\label{PropSasha}
For each $n\in\BN$ there exists an $m\ge n$ such that the composite morphism
$$\kk S_n \hookrightarrow \kk S_m\to\End_{\kk}(P_0),$$
is injective, where the first morphism is induced from the inclusion $S_n<S_m$ (we interpret $S_l$ as the permutation group of $\{1,2,\cdots,l\}$) and $P_0$ is the projective cover of the trivial simple $\kk S_m$-module.
\end{proposition}

It is sufficient to prove Proposition~\ref{PropSasha} for algebraically closed fields, so for the remainder of the section we assume $\kk=\overline{\kk}$. We follow the conventions and notation for Specht and simple modules of symmetric groups as in \cite{James}.

\begin{lemma}\label{LemS1}
For each simple $\kk S_n$-module $D$, there exists $n_1\ge n$ such that $\Ind^{\kk S_{n_1}}_{\kk S_n}D$ has a subquotient isomorphic to a Specht module.
\end{lemma}
\begin{proof}
It suffices to prove that $\Ind^{\kk S_{n_1}}_{\kk S_n}D$ has a simple subquotient labelled by a $p$-core, since the latter are isomorphic to the corresponding Specht modules. More particularly we will prove this for $p$-cores of the form
$$\rho_k=(k^{p-1}, (k-1)^{p-1},\cdots, 1^{p-1})\vdash \frac{1}{2}k(k+1)(p-1).$$
By \cite[Theorem~D]{BK}, it suffices to show that by adding conormal boxes to an arbitrary $p$-regular partition, we can obtain such a $\rho_k$. This is indeed true, since for a given $p$-regular partition $\lambda=(\lambda_1,\lambda_2,\dots)$, iteratively adding the left-most possible box which yields a $p$-regular partition will create $\rho_{\lambda_1}$ via adding conormal boxes.
\end{proof}

\begin{lemma}\label{LemS2}
For each simple $\kk S_n$-module $D$, there exists $n_2\ge n$ such that $\Ind^{\kk S_{n_2}}_{\kk S_n}D$ has a trivial subquotient.
\end{lemma}
\begin{proof}
By Lemma~\ref{LemS1}, it suffices to prove instead that for each Specht module $S^\lambda$, for $\lambda\vdash n$, there exists $n_2\ge n$ such that $\Ind^{\kk S_{n_2}}_{\kk S_n}S^\lambda$ has a trivial subquotient. Since $\Ind^{\kk S_{n_2}}_{\kk S_n}S^\lambda$ has a filtration where the subquotients are given by 
$$\{S^\mu\,|\, \mu\vdash n_2\mbox{ and }\lambda\subseteq\mu\},$$
it is in turn sufficient to show that each partition can be included in a partition which satisfies James' condition in \cite[Theorem~24.4]{James} (which implies that the corresponding Specht module has a trivial submodule). This combinatorial claim is easily verified as follows. Following \cite{James}, for each $a\in\BN$, let $\ell_p(a)$ be the minimal natural number for which $a<p^{\ell_p(a)}$. A partition $\mu$ satisfies James' condition if $\mu_i+1$ is divisible by $p^{\ell_p(\mu_{i+1})}$ for each $i$. For a given partition $\lambda$ of length $l$, we can therefore take the partition $\mu$ of length $l$, by setting $\mu_l=\lambda_1$ and $\mu_i=p^{\ell_p(\mu_{i+1})}-1$ for $i<l$.
\end{proof}

\begin{proof}[Proof of Proposition~\ref{PropSasha}]
Let $m$ be the maximum of the numbers $\{n_2\}$ in Lemma~\ref{LemS2}, where we let $D$ vary over all simple $\kk S_n$-modules. Since the induction of a trivial module has a trivial quotient, it follows that $\Ind^{\kk S_{m}}_{\kk S_n}D$ has a trivial subquotient for each simple $\kk S_n$-module $D$. By adjunction, this means that the projective module $\mathrm{Res}^{\kk S_m}_{\kk S_n}P_0$ is a projective generator and hence faithful as a $\kk S_n$-module.
\end{proof}

\begin{remark}
Analysing the proof of Proposition~\ref{PropSasha} shows that for $m$ we can take for instance $n(p-1)(p^{\ell_p(n)}-1)$.\end{remark}

\bibliographystyle{alpha}

\end{document}